\documentclass[12pt]{amsart}
\usepackage{amssymb}
\usepackage{graphicx}
\usepackage{tikz-cd}
\usepackage[bookmarks=false]{hyperref}

\allowdisplaybreaks

\newtheorem{theorem}[equation]{Theorem}
\newtheorem{lemma}[equation]{Lemma}
\newtheorem{proposition}[equation]{Proposition}
\newtheorem{corollary}[equation]{Corollary}

\newtheorem*{theorem:derhamisomorphism}{Theorem~\ref{T:derhamisomorphism}}
\newtheorem*{theorem:characterization}{Theorem~\ref{T:characterization}}

\theoremstyle{definition}
\newtheorem{definition}[equation]{Definition}

\theoremstyle{remark}
\newtheorem{remark}[equation]{Remark}

\numberwithin{equation}{subsection}

\newcommand{\FF}{\mathbb{F}}
\newcommand{\ZZ}{\mathbb{Z}}

\newcommand{\RR}{\mathbb{R}}
\newcommand{\NN}{\mathbb{N}}
\newcommand{\TT}{\mathbb{T}}

\newcommand{\EE}{\mathbb{E}}
\newcommand{\LL}{\mathbb{L}}
\newcommand{\CC}{\mathbb{C}}

\newcommand{\TTs}{\mathbb{T}_s}
\newcommand{\TTsz}{\mathbb{T}_{s,z}}

\newcommand{\bA}{\mathbf{A}}

\newcommand{\bff}{\mathbf{f}}
\newcommand{\bg}{\mathbf{g}}
\newcommand{\bh}{\mathbf{h}}
\newcommand{\bk}{\mathbf{k}}

\newcommand{\bp}{\mathbf{p}}
\newcommand{\bP}{\mathbf{P}}
\newcommand{\bS}{\mathbf{S}}
\newcommand{\bu}{\mathbf{u}}
\newcommand{\bv}{\mathbf{v}}

\DeclareMathOperator{\Der}{Der}
\DeclareMathOperator{\DR}{DR}
\DeclareMathOperator{\Sol}{Sol}
\DeclareMathOperator{\lcm}{lcm}

\DeclareMathOperator{\Ker}{Ker}
\DeclareMathOperator{\GL}{GL}
\DeclareMathOperator{\Mat}{Mat}

\DeclareMathOperator{\Span}{Span}

\DeclareMathOperator{\Hom}{Hom}
\DeclareMathOperator{\Tr}{Tr}
\DeclareMathOperator{\rank}{rank}
\DeclareMathOperator{\Res}{Res}

\DeclareMathOperator{\ord}{ord}
\DeclareMathOperator{\Ord}{ord_{\infty}}

\newcommand{\ut}{\underline{t}}

\newcommand{\tr}{\mathrm{tr}}
\newcommand{\tpi}{\widetilde{\pi}}

\newcommand{\oFq}{\overline{\FF}_q}

\newcommand{\Fqts}{\FF_q[\ut_s]}

\newcommand{\dnorm}[1]{\lVert #1 \rVert_{\infty}}
\newcommand{\cdnorm}[1]{\lVert #1 \rVert}
\newcommand{\inorm}[1]{{\lvert #1 \rvert}_{\infty}}

\begin{document}

\title[The de Rham isomorphism for Drinfeld modules]{The de Rham isomorphism for Drinfeld modules \\ over Tate algebras}

\author{O\u{g}uz Gezm\.{i}\c{s}}
\address{Department of Mathematics, Texas A{\&}M University, College Station,
TX 77843, U.S.A.}
\email{oguz@math.tamu.edu}

\author{Matthew A. Papanikolas}
\address{Department of Mathematics, Texas A{\&}M University, College Station,
TX 77843, U.S.A.}
\email{papanikolas@tamu.edu}

\thanks{This project was partially supported by NSF Grant DMS-1501362}

\subjclass[2010]{Primary 11G09; Secondary 11R58, 12H10, 12J27, 14F40}

\date{January 30, 2019}

\begin{abstract}
Introduced by Angl\`{e}s, Pellarin, and Tavares Ribeiro, Drinfeld modules over Tate algebras are closely connected to Anderson log-algebraicity identities, Pellarin $L$-series, and Taelman class modules.  In the present paper we define the de Rham map for Drinfeld modules over Tate algebras, and we prove that it is an isomorphism under natural hypotheses.  As part of this investigation we determine further criteria for the uniformizability and rigid analytic triviality of Drinfeld modules over Tate algebras.
\end{abstract}

\keywords{Drinfeld modules, Tate algebras, de Rham map,  uniformizability}

\maketitle

\section{Introduction} \label{S:Intro}

\subsection{Background}
Drinfeld modules over Tate algebras in positive equal characteristic were introduced by Angl\`{e}s, Pellarin, and Tavares Ribeiro in~\cite{AnglesPellarinTavares16}, where they demonstrated that these objects connect the theories of Anderson log-algebraicity identities~\cite{And96}, Pellarin $L$-series~\cite{Pellarin12}, and Taelman class modules~\cite{Taelman10}, \cite{Taelman12}, which are associated to the Carlitz module and more general Goss $L$-series.  Subsequently, Drinfeld modules over Tate algebras have been effective tools for studying special values of positive characteristic $L$-series, modular forms, and Stark units in a number of other contexts (e.g., see \cite{AnglesNgoDacTavares17}, \cite{AnglesPellarinTavares18}, \cite{AnglesTavares17}, \cite{Demeslay14}, \cite{PellarinPerkins18}).

In the present paper we define the de Rham map for Drinfeld modules over Tate algebras, and we investigate conditions under which it is an isomorphism (Theorem~\ref{T:derhamisomorphism}).  For Drinfeld modules over fields of generic characteristic, the de Rham map was first studied by Anderson, Deligne, Gekeler, and Yu (see~\cite{Gekeler89}, \cite{Goss94}, \cite{Yu90}), and Gekeler~\cite[Thm.~5.14]{Gekeler89} gave a proof that it is an isomorphism by way of quasi-periodic functions.  Anderson gave another proof using rigid analytic trivializations and Anderson generating functions (see Goss~\cite[\S 1.5]{Goss94}).

Our investigation into the de Rham isomorphism has led also to criteria for uniformizability of Drinfeld modules over Tate algebras (Theorem~\ref{T:characterization}), which unlike for Drinfeld modules over fields is not guaranteed.  Moreover, we show that under certain conditions, uniformizability is directly related to the existence of rigid analytic trivializations and period lattices of maximal size, much as one finds for Anderson $t$-modules (see~\cite[Thm.~4]{And86}), though with the complication that the base ring of operators has Krull dimension $>1$.

\subsection{The de Rham isomorphism}
Let $\FF_q$ be a finite field with $q$ elements with characteristic $p > 0$.  Let $\theta$, $t_1, \dots, t_s$, $z$ be independent variables over $\FF_q$, let $A = \FF_q[\theta]$, $A[\ut_s] = A[t_1, \dots, t_s]$, $A[\ut_s,z] = A[t_1, \dots, t_s, z]$ be polynomial rings.  We let $K = \FF_q(\theta)$ be the rational function field, $K_\infty = \FF_q((1/\theta))$ its completion at the infinite place, and $\CC_{\infty}$ the completion of an algebraic closure of $K_{\infty}$.  Finally, we let $\TTs$ and $\TTsz$ be Tate algebras on closed unit polydiscs over $\CC_\infty$ in the variables $t_1, \dots, t_s$ and $t_1, \dots, t_s$, $z$.

There is a natural Frobenius twisting automorphism $\tau \colon \TTs \to \TTs$, which is obtained by applying the $q$-th power Frobenius to the coefficients of a given power series in $\TTs$ (see~\S\ref{SS:Tatealgebras}), and for each $n \in \ZZ$, we set $f^{(n)} = \tau^n(f)$ for $f \in \TTs$.  We define the twisted power series ring $\TTs[[\tau]]$ by the rule $\tau f = f^{(1)}\tau$ for $f\in \TTs$, and the twisted polynomial ring $\TTs[\tau]$ is a subring.

Throughout we follow the definitions in Angl\`{e}s, Pellarin, and Tavares Ribeiro~\cite{AnglesPellarinTavares16}, and we define a \emph{Drinfeld $A[\ut_s]$-module}, or a \emph{Drinfeld $A[\ut_s]$-module over $\TTs$}, to be an $\FF_q[\ut_s]$-algebra homomorphism
\begin{equation} \label{E:phiintro}
  \phi \colon A[\ut_s] \to \TTs[\tau],
\end{equation}
where $\TTs[\tau]$ is the twisted polynomial ring in $\tau$ over $\TTs$, such that
\begin{equation} \label{E:phithetaintro}
\phi_\theta = \theta + A_1\tau + \cdots + A_r \tau^r, \quad A_r \neq 0.
\end{equation}
As elements of $\TTs[\tau]$ serve as operators on $\TTs$, we see that $\phi$ induces a left $A[\ut_s]$-module structure on $\TTs$.  One can associate to $\phi$ an \emph{exponential series}
\[
  \exp_\phi = \sum_{i=0}^\infty \alpha_i \tau^i \in \TTs[[\tau]],
\]
defined by $\alpha_0 = 1$ and $\exp_\phi a = \phi_a \exp_\phi$, for $a \in A[\ut_s]$.  We show in Proposition~\ref{P:expentire} that $\exp_\phi$ is an entire operator (see \S\ref{SS:TwistedPolys}), and so there is an induced \emph{exponential function}
\[
  \exp_\phi \colon \TTs \to \TTs.
\]
It should be noted that the function $\exp_\phi$ is not an analytic function on $\TTs$ in the usual sense, as when $r > 0$ it does not have an expansion as a power series in any open disk in $\TTs$.  However, it is open and continuous with respect to the metric on $\TTs$ (see \cite[\S 3.1]{AnglesPellarinTavares16}).
Just as for Drinfeld $A$-modules over $\CC_\infty$, the map $\exp_\phi$ is an $A[\ut_s]$-module homomorphism via the action of $\phi$ on the codomain.  We set $\Lambda_\phi := \ker (\exp_\phi)$ to be the period lattice of~$\phi$.

\begin{remark}
Unlike the situation of Drinfeld $A$-modules over $\CC_\infty$, the exponential function $\exp_\phi$ is not necessarily surjective (e.g., see \cite[\S 3.2]{AnglesPellarinTavares16}).  If $\exp_\phi \colon \TTs \to \TTs$ is surjective, then we say that $\phi$ is \emph{uniformizable}.
\end{remark}

\begin{remark}
For Drinfeld modules over $\CC_\infty$, the exponential function can be expressed as an infinite product over its period lattice (see \cite[Ch.~4]{Goss}, \cite[Ch.~2]{Thakur}). But the obstacle in the case of Drinfeld $A[\ut_s]$-modules is that even though $\Lambda_{\phi}$ is discrete (see \S 3.3 for the precise definition) in the sense that it has no arbitrarily small elements, it is not topologically discrete. Therefore, the construction of an infinite product for $\exp_\phi$ over $\Lambda_{\phi}$ becomes problematic.
\end{remark}

Given a Drinfeld $A[\ut_s]$-module $\phi$ as above, if we assume further that $A_r \in \TTs^{\times}$ in~\eqref{E:phithetaintro}, one can define biderivations and quasi-periodic functions for $\phi$, as in \cite{Brownawell93}, \cite{Brownawell96}, \cite{BP02}, \cite{Gekeler89}, \cite{Gekeler90}, \cite{Gekeler11}, \cite{PR03}, \cite{Yu90}.  Such biderivations are $\FF_q[\ut_s]$-linear homomorphisms $\eta \colon A[\ut_s] \to \tau \TTs[\tau]$ that satisfy $\eta_{ab} = a \eta_b + \eta_a \phi_b$ for all $a$, $b \in A[\ut_s]$.  Associated to each biderivation $\eta$ is a quasi-periodic entire operator $F_{\eta} \in \tau \TTs[[\tau]]$ such that for all $a \in A[\ut_s]$,
\begin{equation}
  F_{\eta} a - a F_{\eta} = \eta_a \exp_{\phi}.
\end{equation}
It follows that the induced map $F_{\eta}|_{\Lambda_{\phi}} \colon \Lambda_{\phi} \to \TTs$ is $A[\ut_s]$-linear.  If we let $\Der(\phi)$ denote the space of all biderivations for $\phi$ and let $\Der_{si}(\phi)$ denote the subspace of all strictly inner biderivations (see~\S\ref{SS:biderivations}), then the \emph{de Rham module} is the left $\TTs$-module $H_{\DR}^*(\phi) := \Der(\phi)/\Der_{si}(\phi)$.  Our main result is the following.

\begin{theorem:derhamisomorphism}
Let $\phi$ be a Drinfeld $A[\ut_s]$-module defined by $\phi_\theta=\theta + A_1\tau + \dots + A_r\tau^r$,
such that \textup{(i)} $A_r\in \TTs^{\times}$, and \textup{(ii)} $\Lambda_{\phi}$ is a free and finitely generated $A[\ut_s]$-module of rank $r$.  We define the $\TTs$-linear de Rham map
\[
\DR \colon H_{\DR}^*(\phi) \to \Hom_{A[\ut_s]}(\Lambda_{\phi},\TTs)
\]
by $\DR([\eta])=F_{\eta}|_{\Lambda_{\phi}}$. Then $\DR$ is an isomorphism of left $\TTs$-modules.
\end{theorem:derhamisomorphism}

The set-up and proof of this theorem occupy a major portion of the paper.  One considerable obstacle is that significant parts of the proof of Gekeler~\cite[Thm.~5.14]{Gekeler89}, which would be well-suited from first principles, do not extend to Drinfeld $A[\ut_s]$-modules over $\TTs$, because certain properties do not extend from $\CC_{\infty}$ to $\TTs$ (e.g., the lack of product expansion for $\exp_\phi$ above).  Instead we adopt a combined approach with that of Anderson given in \cite[\S 1.5]{Goss94}, which required us to develop the theory of rigid analytic trivializations and Anderson generating functions for Drinfeld $A[\ut_s]$-modules.  To do this we adapt constructions that are originally due to Anderson for Drinfeld $A$-modules over $\CC_{\infty}$ (in unpublished work), but which are treated in ~\cite{HartlJuschka16} by Hartl and Juschka.

\subsection{Uniformizability criteria}
In proving Theorem~\ref{T:derhamisomorphism} it becomes apparent that the de Rham map being an isomorphism is interconnected with several other properties of the Drinfeld $A[\ut_s]$-module $\phi$, namely uniformizability and rigid analytic triviality.  The idea of rigid analytic triviality goes back to Anderson in~\cite{And86}, and while we sketch the definition in our context here, it is defined fully in~\S\ref{SS:rat}.

Continuing with the definition of Drinfeld $A[\ut_s]$-module $\phi$ from~\eqref{E:phiintro} and~\eqref{E:phithetaintro}, we assume further that $A_r \in \TTs^{\times}$.  We let $\sigma = \tau^{-1}$, and letting $H(\phi) = \TTs[\sigma]$, we give $H(\phi)$ the structure of a $\TTs[z]$-module by setting for $h \in H(\phi)$,
\[
  z \cdot h := h \phi_{\theta}^* = h \bigl( \theta + A_1^{(-1)}\sigma + \cdots + A_r^{(-r)} \sigma^r \bigr).
\]
We call $H(\phi)$ a \emph{Frobenius module} (in the sense of \cite[\S 2.2]{CPY18}), and we show that $H(\phi)$ is free of rank~$r$ as a $\TTs[z]$-module, with basis $1$, $\sigma, \dots, \sigma^{r-1}$ (Lemma~\ref{L:Hphibasis}). For any $\mathbb{F}_q$-algebra $R$, let $\Mat_{r}(R)$ be the set of $r\times r$-matrices with entries in $R$, and $\GL_{r}(R)$ be the set of invertible $r\times r$-matrices in $\Mat_{r}(R)$. With respect to the basis $\{1,\sigma,\dots,\sigma^{r-1}\}$, there is a matrix $\Phi \in \Mat_r(\TTs[z])$ such that $\Phi$ represents multiplication by $\sigma$ on $H(\phi)$ (see \S\ref{SS:zframes}), and a \emph{rigid analytic trivialization} for $\phi$ is a matrix $\Psi \in \GL_r(\TTs\{ z/\theta\})$ satisfying
\[
  \Psi^{(-1)} = \Phi\Psi,
\]
where $\TTs\{ z/\theta \}$ is the subalgebra of $\TTsz$ consisting of functions that converge as far out as $|z|_{\infty} \leqslant |\theta|_{\infty}$ in the variable $z$.  Finally, we let $\phi[\theta] = \{ f \in \TTs \mid \phi_{\theta}(f) = 0 \}$, the $\theta$-torsion of $\phi$ in $\TTs$.  The connections among these objects are as follows.

\begin{theorem:characterization}
Let $\phi$ be a Drinfeld $A[\ut_s]$-module of rank $r$ defined by
\[
\phi_\theta=\theta + A_1\tau + \dots + A_r\tau^r
\]
such that \textup{(i)} $A_r \in \TTs^{\times}$, and \textup{(ii)} $\Lambda_{\phi}$ is a free and finitely generated $A[\ut_s]$-module.  Then the following are equivalent.
\begin{enumerate}
\item[(a)] $\Lambda_{\phi}$ is free of rank $r$ over $A[\ut_s]$.
\item[(b)] $\phi$ has a rigid analytic trivialization.
\item[(c)] The de Rham map $\DR$ is an isomorphism.
\item[(d)] $\phi$ is uniformizable, and $\phi[\theta]$ is free of rank $r$ over $\FF_q[\ut_s]$.
\end{enumerate}
\end{theorem:characterization}

\begin{remark}
(a)~By the Quillen-Suslin Theorem (see \cite[Thm.~XXI.3.7]{Lang}), if $\Lambda_{\phi}$ is a finitely generated and projective $A[\ut_s]$-module, then it is also a free module over $A[\ut_s]$.  Thus Theorem~\ref{T:characterization} is no more general if we allow $\Lambda_{\phi}$ to be finitely generated and projective.  (b)~Angl\`es, Pellarin, and Tavares Ribeiro~\cite[Prop.~6.2, Rem.~6.3]{AnglesPellarinTavares16} show that in the rank~$1$ case, $\phi$ being uniformizable is equivalent to $\Lambda_{\phi}$ being free of rank~$1$ (see Proposition~\ref{P:char}).  On the other hand, it is conceivable for higher ranks that $\phi$ could be uniformizable but that $\Lambda_{\phi}$ is not free.  (c)~The notion that the de Rham map being an isomorphism should be equivalent to $\phi$ being uniformizable with $\Lambda_\phi$ of maximal rank was originally introduced to us by Brownawell.
\end{remark}

\subsection{Outline of the paper}
The paper is organized as follows.  We review fundamental information about Tate algebras and associated $\tau$-difference equations in~\S\ref{S:Notation}.  In~\S\ref{S:DrinfeldTate} we review the theory of Drinfeld modules over Tate algebras as introduced in~\cite{AnglesPellarinTavares16}, we develop properties of their exponentials and logarithms, and we discuss Anderson generating functions.  In~\S\ref{S:Frobenius} we study Frobenius modules for Drinfeld $A[\ut_s]$-modules and drawing on arguments of Anderson as given in Hartl and Juschka~\cite{HartlJuschka16}, we explore the image of the exponential function.  We define rigid analytic trivializations and prove their connections to the surjectivity of the exponential.  The theory of biderivations and the de Rham map are introduced in~\S\ref{S:deRhamiso}, and the proof of Theorem~\ref{T:derhamisomorphism} occupies~\S\ref{S:deRhamisoProof}.  In~\S\ref{S:Uniformizability} we discuss and prove the uniformizability criteria of Theorem~\ref{T:characterization}.  Finally, in~\S\ref{S:Applications} we consider various applications and examples.

\subsection*{Acknowledgments}
The authors thank the referee for carefully reading our manuscript and for making several useful suggestions.

\section{Notation and preliminaries} \label{S:Notation}

\subsection{Table of notation}
The following notation will be used throughout:

\begin{tabular}{p{1truein}@{\hspace{5pt}$=$\hspace{5pt}}p{4.5truein}}
$\FF_q$ & finite field with $q=p^m$ elements. \\
$\theta$, $t_1, \dots, t_s$, $z$ & independent variables over $\FF_q$. \\
$A$ & $\FF_q[\theta]$, the polynomial ring in $\theta$ over $\FF_q$. \\
$K$ & $\FF_q(\theta)$, the field of rational functions in $\theta$ over $\FF_q$. \\
$K_{\infty}$& $\FF_q((1/\theta))$, the $\infty$-adic completion of $K$. \\
$\CC_{\infty}$ & the completion of an algebraic closure of $K_{\infty}$.\\
$\ut_s$ & abbreviation for the list of variables $t_1, \dots, t_s$.\\
$A[\ut_s]$ & $A[t_1,\dots,t_s]$, the polynomial ring in $\theta$, $t_1, \dots, t_s$ over $\FF_q$. \\
$\TTs$ & Tate algebra on closed unit polydisc with parameters $t_1, \dots ,t_s$ and coefficients in $\CC_{\infty}$.\\
$\TTsz$ & Tate algebra on closed unit polydisc with parameters $t_1, \dots, t_s$, $z$ and coefficients in $\CC_{\infty}$. \\
$\LL_s$ & the fraction field of $\TTs$. \\
$\LL_{s,z}$ &  the fraction field of $\TTsz$.
\end{tabular}

\subsection{Tate algebras} \label{SS:Tatealgebras}
We let $\inorm{\,\cdot\,}$ denote the $\infty$-adic norm on $\CC_\infty$, normalized so that $\inorm{\theta} = q$, and we take $\ord_{\infty}$ to be the associated valuation such that $\ord_{\infty}(\theta) = -1$.  For $s \geqslant 1$ and for a power series $f=\sum a_{\nu_1\cdots \nu_s}t_1^{\nu_1}\cdots t_s^{\nu_s} \in \CC_{\infty}[[t_1, \dots, t_s]]$, we abbreviate $f$ as $\sum a_{\nu}\ut_s^{\nu}$, where $\nu$ is an $s$-tuple of non-negative integers.  For such an $s$-tuple $\nu$, we let $|\nu| := \nu_1 + \cdots + \nu_s$.  The Tate algebra $\TTs$ is then defined by
\[
  \TTs := \biggl\{ \sum_{\nu} a_{\nu} \ut_s^{\nu} \in \CC_{\infty}[[t_1, \dots, t_s]] \biggm| \inorm{a_{\nu}} \to 0\ \textup{as}\ |\nu| \to \infty \biggr\}.
\]
For foundations on Tate algebras we appeal to results in~\cite[Ch.~2--3]{FresnelvdPut}.  As is customary we define a Frobenius twisting automorphism $\tau\colon \TTs \to \TTs$ by
\[
 \tau\biggl(\sum a_{\nu}\ut_s^{\nu}\biggr) := \sum a_{\nu}^{q} \ut_s^{\nu},
\]
and we let $\sigma:=\tau^{-1}$.  For $f\in \TTs$ and $n \in \ZZ$, the $\emph{$n$-fold twist of $f$}$ is defined to be
\[
f^{(n)} := \tau^{n}(f).
\]
For a matrix $A=(A_{ij})\in$ Mat$_{r}(\TTs)$, we define $A^{(n)}:=(A^{(n)}_{ij})$.

We define the Gauss norm $ \|\cdot\|_\infty $  on $\TTs$ by setting for $f = \sum a_{\nu} \ut_s^{\nu} \in \TTs$,
\[
\|f\|_\infty  := \sup\{\inorm{a_{\nu}} \mid \nu \in \ZZ_{\geqslant 0}^s \},
\]
and we denote its associated valuation by $\ord$ for which $\ord_{\infty}(f) = \min\{ \ord_\infty(a_{\nu}) \mid \nu \in \ZZ_{\geqslant 0}^s\}$.  We note that $\TTs$ is complete with respect to $\|\cdot\|_\infty$.  We similarly denote the Gauss norm and valuation on $\TTsz$.

We will need Tate algebras that converge on disks of more general radii. For $c \in \TTs^{\times}$, set
\[
\TTs \{ z/c \}= \biggl\{ \sum_{i=0}^{\infty} a_i z^i \in \TTs[[z]] \biggm| \|c\|_\infty^{i} \cdot \|a_i\|_\infty \to 0 \textup{\ as\ } i \to \infty \biggr\}.
\]
We define the norm $\cdnorm{f}_c=\cdnorm{\sum a_i z^i}_c := \sup_i \{ \|c\|_\infty^{i} \cdot \|a_i\|_\infty \}$, and it follows that $\TTs\{z/c \}$ is complete with respect to $\cdnorm{\,\cdot\,}_c$.  Given $\Upsilon \in \Mat_{r \times \ell}(\TTs\{ z/c \})$, put $\cdnorm{\Upsilon}_c=\max_{i,j} \{ \cdnorm{\Upsilon_{ij}}_c \}$.

\begin{lemma}\label{L:limit}
For any $f\in \TTs$ with $\dnorm{f} \leqslant 1$, there is a positive integer $\ell$ so that with respect to $\dnorm{\,\cdot\,}$ we have $\lim_{n \to \infty} f^{(n\ell)} \in \oFq[\ut_s]$.  Also, $\dnorm{f}=1$ if and only if $\lim_{n \to \infty} f^{(n\ell)} \neq 0$.
\end{lemma}

\begin{proof}
Let $f=\sum a_\nu \ut_s^{\nu}$.  If $\dnorm{f} < 1$, then taking $\ell = 1$ suffices, since in this case the limit easily goes to $0$.  If $\dnorm{f} = 1$, then there exist only finitely many multi-indices $\nu_1, \dots, \nu_m$ whose corresponding coefficients have norm $1$. By \cite[Lem.~2.2.6]{P08}, for each $j$ there exists $\ell_{j}>0$ and $c_j \in \oFq^\times$ such that $\lim_{n \to \infty} a_{\nu_j}^{(n \ell_j)} =c_{j}$.  If we then take $\ell = \lcm(\ell_1, \dots \ell_m)$, it follows that
\[
\lim_{n \to \infty} f^{(n\ell)} = c_1 \ut_s^{\nu_1} + \dots + c_m \ut_s^{\nu_m} \in \oFq[\ut_s],
\]
which is necessarily non-zero.
\end{proof}

Given a $\TTs$-module $W$, we recall that a faithful norm $\cdnorm{\,\cdot\,}_W \colon W \to \RR_{\geqslant 0}$ on $W$ satisfies the following properties (see \cite[\S 2.1]{BGR}): (i) $\cdnorm{w}_W=0$ if and only if $w=0$; (ii) $\cdnorm{w_1 + w_2}_W \leqslant \max\{\cdnorm{w_1}_W, \cdnorm{w_2}_W \}$ for all $w_1$, $w_2\in W$; and (iii) $\cdnorm{fw}_W = \dnorm{f}\cdot \cdnorm{w}_W$ for all $f\in \TTs$ and $w\in W$.  Since $\TTs$ has a complete valued norm with respect to $\dnorm{\,\cdot\,}$, every faithfully normed and finitely generated $\TTs$-module $W$ manifests a complete faithful norm that is essentially unique, as we see in the following lemma.

\begin{lemma}[{\cite{BGR}, see \S 2.1.8, Cor.~4; \S 3.7.3, Prop.~3}] \label{L:EquivNorms}
Let $W$ be a finitely generated $\TTs$-module.  Then there exists a complete faithful norm $\cdnorm{\,\cdot\,}_W$ on $W$, and for any other complete faithful norm $\cdnorm{\,\cdot\,}_W'$ on $W$, there exist $C$, $C' >0$ so that for all $w \in W$,
\[
C \cdnorm{w}_W' \geqslant \cdnorm{w}_W \geqslant C' \cdnorm{w}_W'.
\]
\end{lemma}

As a primary example, we provide $\Mat_{r \times \ell}(\TTs)$ a complete faithful norm $\cdnorm{\,\cdot\,}_{r\times \ell}$ by setting $\cdnorm{M}_{r \times \ell} := \sup \{ \dnorm{M_{ij}} \}$ for any $M=(M_{ij})\in \Mat_{r \times \ell}(\TTs)$.  By an abuse of notation we will also denote $\cdnorm{\,\cdot\,}_{r\times \ell} = \dnorm{\,\cdot\,}$.  All of the preceding considerations extend to $\TT_{s,z}$ in the obvious manner.

\subsection{Anderson-Thakur elements} \label{SS:ATelements}
We now recall special Anderson-Thakur type elements of $\TTs^{\times}$ due to Angl\`{e}s, Pellarin, and Tavares Ribeiro~\cite{AnglesPellarinTavares16}, which generalize the Anderson-Thakur function $\omega$ from~\cite{AndThak90}. For $\alpha \in \TTs^{\times}$, we construct an element $\omega(\alpha)$ in $\TTs^{\times}$ as follows.  By the invertibility of $\alpha$, there exists $x\in \CC_{\infty}^{\times}$ such that $\dnorm{x-\alpha} < \dnorm{\alpha}$, and so we have
\[
\biggl\lVert \frac{x^{q^i}}{\tau^i(\alpha)} - 1 \biggr\rVert_{\infty}
= \biggl\lVert \frac{x^{q^i} - \tau^i(\alpha)}{\tau^i(\alpha)} \biggr\rVert_{\infty}
< \biggl\lVert \frac{\tau^i(x-\alpha)}{\tau^i(\alpha)} \biggr\rVert_{\infty} \to 0
\quad \textup{as} \quad i \to \infty.
\]
Thus $\prod_i x^{q^i}/\tau^i(\alpha)$ converges in $\TTs^{\times}$.  Fixing an element $\gamma \in \CC_{\infty}$ with $\gamma^{q-1} = x$, we define
\begin{equation}\label{E:omega}
  \omega(\alpha) := \gamma \prod_{i=0}^{\infty} \frac{x^{q^i}}{\tau^i(\alpha)} \in \TTs^{\times},
\end{equation}
and we then see that
\begin{equation} \label{E:tauomegaalpha}
\tau(\omega(\alpha))=\alpha\omega(\alpha).
\end{equation}
Although it appears that $\omega(\alpha)$ depends on the choice of $x$, the limit is uniquely defined up to a scalar multiple from $\FF_q^{\times}$.  Moreover, for $\alpha_1$, $\alpha_2 \in \TTs^{\times}$, we have
\begin{equation}\label{E:property}
\omega(\alpha_1\alpha_2)=c\omega(\alpha_1)\omega(\alpha_2)
\end{equation}
for some $c \in \FF_q^{\times}$ which depends on the choice of $(q-1)$-st roots.

\subsection{Twisted polynomials and power series} \label{SS:TwistedPolys}
The ring $\TTs[\tau]$ operates on $\TTs$ by setting for $\Delta = a_r \tau^r + \dots + a_0 \in \TTs[\tau]$ and $f \in \TTs$,
\[
  \Delta(f) = a_r \tau^r(f) + \dots + a_1\tau(f) + a_0 f = a_r f^{(r)} + \dots + a_1 f^{(1)} + a_0 f.
\]

In general $\TTs[[\tau]]$ does not operate on $\TTs$ because the desired sum may fail to converge, but as in~\cite[\S 3.3]{AnglesPellarinTavares16}, we can define entire operators in $\TTs[[\tau]]$ as follows.  A function $f = \sum a_\nu \ut_s^{\nu} \in \TTs$ is called an \emph{entire function} if $\lim_{|\nu| \to \infty} \ord_{\infty}(a_\nu)/|\nu| = +\infty$.  In this case $f$ then converges on all of $\CC_\infty^s$.  We let $\EE_s \subseteq \TTs$ denote the subring of all entire functions, which contains $\CC_\infty[\ut_s]$ as a subring, and we note that $\EE_s$ is invariant under $\tau$.  Now consider an element $F = \sum_i F_i \tau^i \in \EE_s[[\tau]]$.
If
\[
  \lim_{i \to \infty} q^{-i} \cdot \ord_{\infty}(F_i) = +\infty,
\]
we say that $F$ is an \emph{entire operator}.  In this case for $f \in \TTs$, we have $F(f)  = \sum_{i=0}^\infty F_i f^{(i)} \in \TTs$.

\begin{lemma}[{Angl\`{e}s, Pellarin, Tavares Ribeiro~\cite[Lem.~3.9]{AnglesPellarinTavares16}}]
Let $F \in \EE_s[[\tau]]$ be an entire operator.  Then $F(\EE_s) \subseteq \EE_s$.
\end{lemma}

\subsection{$\tau$-difference equations}
We recall some properties of $\tau$-difference equations, and for more detailed information the reader is directed to \cite[\S 4.1]{P08}, \cite{vdPutSinger97}. Recall also that $\LL_s$ is the fraction field of $\TT_s$. For $\Delta \in \TTs[\tau]$ we set
\[
\Sol_s(\Delta) := \{ f \in \LL_{s} \mid \Delta(f)=0 \}.
\]
If $R \subseteq \LL_s$ is a subring with $\tau(R)\subseteq R$, then we set $\Sol_s(\Delta,R) := \Sol_s(\Delta) \cap R$.
Likewise we similarly define $\Sol_{s,z}(\Delta)$ and $\Sol_{s,z}(\Delta,R)$ for $\Delta \in \TTsz[\tau]$.  For any subring $R \subseteq \LL_s$ that is invariant under twisting we set
\[
  R^\tau := \Sol_s(\tau -1,R) = \{ f \in R \mid \tau(f) = f \}
\]
to be the $\FF_q$-subalgebra fixed by $\tau$.  Then following lemma is fundamental.

\begin{lemma}[{see \cite[Lem.~2.2]{Demeslay14}}] \label{L:fixed}
We have $\TTs^{\tau} = \FF_q[\ut_s]$ and $\LL_s^{\tau} = \FF_q(\ut_s)$. Similarly, $\TTsz^{\tau} = \FF_q[\ut_s,z]$ and $\LL_{s,z}^{\tau} = \FF_q(\ut_s,z)$.
\end{lemma}

By this lemma we see that for any $\Delta \in \TTs[\tau]$, the space $\Sol_s(\Delta)$ is an $\FF_q(\ut_s)$-vector space.  As is well-understood in this situation, the dimension of this vector space is bounded by the degree in~$\tau$ of $\Delta$ (see \cite[Lem.~5.7]{AnglesPellarinTavares16}, \cite[Cor.~4.1.5]{P08}, \cite[\S 1.2]{vdPutSinger03}).

\begin{lemma}\label{L:Solution}
For $\Delta \in \TTs[\tau]$, suppose that $\deg_\tau \Delta = r$.  Then $\dim_{\FF_q(\ut_s)} \bigl( \Sol_s(\Delta) \bigr) \leqslant r$.
\end{lemma}

Now for $\alpha \in \TTs^{\times}$, we let $\Delta_\alpha = \alpha \tau - 1$.  We see from \eqref{E:tauomegaalpha} that $\omega(\alpha)^{-1} \in \Sol_s(\Delta_{\alpha})$, and moreover, Lemma~\ref{L:Solution} implies that $\Sol_s(\Delta_\alpha) = \omega(\alpha)^{-1} \cdot \FF_q(\ut_s)$.  In fact we have a similar result for $\Sol_s(\Delta_\alpha,\TTs)$.

\begin{proposition}[{Angl\`{e}s, Pellarin, Tavares Ribeiro \cite[Rem.~6.3]{AnglesPellarinTavares16}}]\label{P:taudiff}
Let $\alpha \in \TTs^{\times}$. Then
\[
\Sol_s(\Delta_\alpha,\TTs) = \frac{1}{\omega(\alpha)} \cdot \Fqts.
\]
\end{proposition}

\section{Drinfeld modules over Tate algebras} \label{S:DrinfeldTate}

\subsection{Drinfeld $A[ \protect \ut_s]$-modules}
The main objects of study in this paper are \emph{Drinfeld $A[\ut_s]$-modules over $\TTs$}, which were defined by Angl\`{e}s, Pellarin, and Tavares Ribeiro~\cite{AnglesPellarinTavares16} for the purposes of expressing new results on Pellarin $L$-values, Taelman class modules, and log-algebraicity identities in operator-theoretic language.  Such a Drinfeld module, of rank $r \geqslant 1$, is an $\FF_q[\ut_s]$-algebra homomorphism
\begin{equation}
  \phi \colon A[\ut_s] \to \TTs[\tau]
\end{equation}
determined uniquely by
\begin{equation} \label{E:Drinfeld}
  \phi_\theta = \theta + A_1 \tau + \dots + A_r \tau^r, \quad  A_r \neq 0.
\end{equation}
If the parameters $A_1, \dots, A_r$ are all in $\CC_\infty$, then we say that $\phi$ is a \emph{constant} Drinfeld $A[\ut_s]$-module, and indeed in this case $\phi$ is simply a traditional Drinfeld $A$-module.

A morphism $u \colon \phi \to \psi$ of Drinfeld $A[\ut_s]$-modules is a twisted polynomial $u \in \TTs[\tau]$ such that for all $f \in A[\ut_s]$, we have $u \phi_f = \psi_f u$.  This is equivalent to satisfying the single equation $u \phi_\theta = \psi_\theta u$.  If $u \neq 0$, then $\phi$ and $\psi$ must have the same rank.  Furthermore, $u$ is an isomorphism if $u \in \TTs^{\times}$, and in this case we will write $\phi \cong \psi$.

\subsection{Exponential and logarithm series} \label{SS:ExpLog}
For a Drinfeld $A[\ut_s]$-module $\phi$, we can define an exponential series for $\phi$ in the usual way.  We take
\begin{equation}
  \exp_\phi = \sum_{i=0}^\infty \alpha_i \tau^i \in \TTs[[\tau]],
\end{equation}
subject to the conditions that $\alpha_0 = 1$ and
\begin{equation} \label{E:exp}
  \exp_\phi a = \phi_{a} \exp_\phi, \quad \forall\, a \in A[\ut_s].
\end{equation}
It then suffices to show that there exists a unique normalised formal series $\exp_{\phi}$ such that $\phi_{\theta}\exp_{\phi}=\exp_{\phi}\theta$. The usual argument for constant Drinfeld modules~\cite[Prop.~4.6.7]{Goss}, \cite[Thm.~2.4.2]{Thakur}, shows that this functional equation produces a recursion on $\{ \alpha_i \}$ that uniquely determines $\exp_\phi$.

Now we recall some terminology from \cite[\S 3]{EP13}, \cite[\S 5]{EP14}. For $S \subseteq \ZZ$ and $j\in \ZZ$, we define $S+j:=\{k+j \mid k\in S \}$. For $r$, $i \in \ZZ_+$ we define the set of \emph{shadowed partitions} $P_r(i)$ as follows.  We let $P_r(i) \subseteq \{ (S_1, \dots, S_r) \mid S_k \subseteq \{0, 1, \ldots, i-1\} \}$ be the set of $r$-tuples $(S_1, \dots, S_r)$ such that $\{S_k + j \mid 1\leqslant k \leqslant r,\, 0\leqslant j \leqslant k-1 \}$ forms a partition of $\{0,1,\ldots, i-1\}$.  We set $P_r(0):=\{ \emptyset \}$.  For $\bS = (S_1, \dots, S_r) \in P_r(i)$ and $\bA = (A_1, \dots, A_r) \in \TTs^r$, we set
\[
  \bA^{\bS} := \prod_{k=1}^r A_i^{S_k}, \quad A_k^{S_k} := \prod_{j \in S_k} \tau^j (A_k),
\]
and take $\bA^{\emptyset} = 1$.  Using $[i] = \theta^{q^i} - \theta$, we set $D_{i}(\bS) := \prod_{k \in S_1 \cup \cdots \cup S_r} [i-k]^{q^k}$.

\begin{proposition} \label{P:expentire}
For a Drinfeld $A[\ut_s]$-module $\phi$ the exponential $\exp_\phi$ is an entire operator.
\end{proposition}

\begin{proof}
Applying the same methods as in the proof of \cite[Thm.~3.1]{EP13}, we see that
\[
\alpha_i = \sum_{\bS \in P_r(i)} \frac{\bA^{\bS}}{D_i(\bS)}.
\]
Let $\xi=\inf \{\ord_{\infty}(A_k) \mid 1\leqslant k \leqslant r\}$.  Then by \cite[Eq.~28]{EP13}, we find that if $\xi \geqslant 0$,
\begin{equation}\label{bound4}
\ord_{\infty}(\alpha_i) \geqslant  \frac{q^i-1}{q^r-1}\cdot \xi + \frac{iq^i}{r} \geqslant \frac{iq^i}{r},
\end{equation}
and if $\xi < 0$,
\begin{equation}\label{bound5}
\ord_{\infty}(\alpha_i)\geqslant  \frac{q^i-1}{q-1}\cdot \xi + \frac{iq^i}{r} \geqslant q^i \biggl(\frac{i}{r} + \frac{\xi}{q-1}\biggr).
\end{equation}
Together~\eqref{bound4} and~\eqref{bound5} imply that $\lim_{i \to \infty} \ord_{\infty}(\alpha_i)/q^i = \infty$, as desired.
\end{proof}

Since the constant term of $\exp_\phi$ is $1$, $\exp_\phi$ is a unit in $\TTs[[\tau]]$, and its inverse
\begin{equation} \label{E:Log}
  \log_{\phi} = \sum_{i=0}^{\infty} \beta_i \tau^i \in \TTs[[\tau]],
\end{equation}
is the \emph{logarithm series}, which satisfies $\beta_0 = 1$ and $\log_{\phi}\phi_{a}=a\log_{\phi}$, for all $a \in A[\ut_s]$.  In general $\log_\phi$ is not an entire operator.  For $\bS \in P_r(n)$, if we set $L(\bS) := \prod_{k=1}^r \prod_{j\in S_k}(-[j+k])$, then using the same methods in \cite[Thm.~3.3]{EP13}, we find that
\begin{equation} \label{E:logcoeffs}
  \beta_i = \sum_{\bS \in P_r(i)} \frac{\bA^{\bS}}{L(\bS)}.
\end{equation}

\subsection{Uniformizability}
The entire operator $\exp_\phi \in \TTs[[\tau]]$ induces an $\FF_q[\ut_s]$-linear function
\begin{equation}
  \exp_\phi \colon \TTs \to \TTs,
\end{equation}
which we call the \emph{exponential function} of $\phi$.

\begin{lemma}[{cf.~Hartl-Juschka \cite[Lem.~5.3]{HartlJuschka16}}]\label{L:iso}
Given our Drinfeld $A[\ut_s]$-module defined in \eqref{E:Drinfeld}, there exists $\varepsilon_{\phi} > 0$ such that the open ball $\{f\in \TTs \mid \|f\|_\infty < \varepsilon_{\phi} \}\subseteq \TTs$ is mapped $\|\,\cdot\,\|_{\infty}$-isometrically by $\exp_{\phi}$ to itself.
\end{lemma}

\begin{proof}
Since $\exp_{\phi}=\sum \alpha_i\tau^i$ is everywhere convergent, we have that $C:=\sup\{ \|\alpha_j\|_\infty \mid j \geqslant 1 \}$ is bounded.  Then take $\varepsilon_{\phi} := (\sup\{ \|\alpha_j\|_\infty^{1/(q^j-1)} \mid j \geqslant 1 \})^{-1}$. For $f\in \TTs$ with $\|f\|_\infty < \varepsilon_{\phi}$, we have
\[
\|f-\exp_{\phi}(f)\|_{\infty}\leqslant \sup_{j=1}^{\infty} \|\alpha_j\|_\infty\|f\|_\infty^{q^j-1}\dnorm{f} < \sup_{j=1}^{\infty}  \|\alpha_j\|_\infty \cdot \varepsilon_{\phi}^{q^j-1} \cdot \|f\|_\infty =\|f\|_\infty.
\]
If $\|f\|_\infty< \| \exp_{\phi}(f)\|_\infty$, then  $\|\exp_{\phi}(f)\|_\infty=\|f -\exp_{\phi}(f)\|_\infty <\|f\|_\infty$, a contradiction.  If $\|f\|_\infty>\| \exp_{\phi}(f)\|_\infty$, then  $\|f\|_\infty=\|f -\exp_{\phi}(f)\|_\infty <\|f\|_\infty$, also a contradiction.  Thus $\|f\|_\infty=\|\exp_{\phi}(f)\|_\infty$.
\end{proof}

\begin{definition}
Let $\mathfrak{m}_{\TT_s}$ be the set of elements $f\in \TTs$ such that $\dnorm{f}<1$. An $A[\ut_s]$-submodule $M$ of $\TTs$ is called a discrete $A[\ut_s]$-module if there exists $n\geq 1$ such that $M \cap \mathfrak{m}^n_{\TT_s} = \{ 0\}$.
\end{definition}

\begin{corollary}
The period lattice $\Lambda_{\phi}=\{ \lambda \in \TTs \mid \exp_{\phi}(\lambda)=0\} \subseteq \TTs$ of the Drinfeld $A[\ut_s]$-module $\phi$ is discrete. In particular, there exists $\epsilon > 0$ such that $\Lambda_{\phi} \cap \{x\in \TTs \mid \|x\|_{\infty} < \epsilon \}=\{0\}$.
\end{corollary}

Following Anderson \cite{And86}, we say that a Drinfeld $A[\ut_s]$-module $\phi$ is \emph{uniformizable} if the induced exponential function $\exp_{\phi} \colon \TTs \to \TTs$ is surjective.  Unlike constant Drinfeld modules over $\CC_{\infty}$, the function $\exp_\phi \colon \TTs \to \TTs$ may not be surjective.  For example, it is shown in~\cite[\S 3.2]{AnglesPellarinTavares16} that the exponential function of the Drinfeld $A[t_1]$-module $\phi$ defined by $\phi_{\theta}=\theta + t_1\tau$ is not surjective.  For the rank~$1$ case, the characterization of uniformizable Drinfeld $A[\ut_s]$-modules is given by the following proposition.

\begin{proposition}[{Angl\`{e}s, Pellarin, Tavares Ribeiro \cite[Prop.~6.2]{AnglesPellarinTavares16}}] \label{P:char}
Let $\phi$ be a Drinfeld $A[\ut_s]$-module of rank $1$ over $\TTs$ defined by $\phi_{\theta}=\theta + \alpha \tau$. The following are equivalent.
\begin{enumerate}
\item[(a)] $\phi$ is uniformizable.
\item[(b)] $\Sol_s(\Delta_\alpha,\TTs) = \Sol_s(\alpha \tau -1,\TTs ) \neq 0$.
\item[(c)] $\alpha \in \TTs^{\times}$.
\item[(d)] $\phi$ is isomorphic to the Carlitz module over $\TTs$.
\end{enumerate}
\end{proposition}

\subsection{Anderson generating functions} \label{SS:AGF}
We continue with our Drinfeld $A[\ut_s]$-module $\phi$ of rank $r$ defined as in~\eqref{E:Drinfeld}. For $\lambda \in \TTs$ we define the Anderson generating function $f_{\lambda}(z)$ as
\begin{equation}\label{E:AGF}
f_{\lambda}(z) = \sum_{n=0}^{\infty} \exp_{\phi}\biggl(\frac{\lambda}{\theta^{n+1}}\biggr)z^n \in \TTs[[z]].
\end{equation}
We fix $\Delta_\phi = A_r \tau^r + \dots + A_1 \tau - (z - \theta) \in \TTs[z][\tau]$, and we have the following structural result, which is due to Pellarin for the constant or even isotrivial Drinfeld $A[\ut_s]$-modules (see \S 8.3 for the definition of isotrivial Drinfeld $A[\ut_s]$-modules). But the proof is essentially the same for any Drinfeld $A[\ut_s]$-modules.

\begin{proposition}[{Pellarin~\cite[\S 4.2]{Pellarin08}}] \label{P:ResSol}
Let $\exp_{\phi}=\sum_i \alpha_i\tau^i \in \TTs[[\tau]]$ be the exponential series of a Drinfeld $A[\ut_s]$-module $\phi$.
\begin{enumerate}
\item[(a)] For $\lambda \in \TTs$, we have
\[
f_{\lambda}(z)= \sum_{n=0}^{\infty}\frac{\alpha_n \lambda^{(n)}}{\theta^{q^n}-z}\in \TTsz.
\]
\item[(b)] For any $j\geqslant 1$ and $\lambda \in \TTs$, $f_{\lambda}^{(j)}(z)\in \TTs\{z/\theta\}$.
\item[(c)] As a function of $z$, $f_{\lambda}(z)$ has poles at the points $z = \theta^{q^n}$, $n=0,1,\dots$, with residues $\Res_{z=\theta^{q^n}}f_{\lambda}(z) = -\alpha_n \lambda^{(n)}$.  In particular, $\Res_{z=\theta}f_{\lambda}(z)=-\lambda$.
\item[(d)] If $\exp_{\phi}(\lambda) = \xi$, then $\Delta_{\phi}(f_{\lambda}(z)) = \xi$.
\end{enumerate}
\end{proposition}

\begin{proof}
The proofs of part (a), (c), and (d) follow the same argument in \cite[\S 4.2]{Pellarin08}, and we give only a proof of (b). By Lemma \ref{L:iso}, for arbitrarily large $n\in \NN$, we have that $\dnorm{\lambda/\theta^{n}}=\dnorm{\exp_{\phi}(\lambda/\theta^{n})}$. Therefore, for any $j\geqslant 1 $ and large $n$, we have that
\[
\inorm{\theta^{n}} \biggl\lVert \exp_{\phi}\biggl(\frac{\lambda}{\theta^{n+1}}\biggr)^{(j)} \biggr\rVert_{\infty} = \inorm{\theta^{n}}\biggl\lVert \frac{\lambda^{(j)}}{\theta^{q^j(n+1)}} \biggr\rVert_{\infty} = \biggl\lVert \frac{\lambda^{(j)}}{\theta^{n(q^j-1)+q^j}} \biggr\rVert_{\infty},
\]
and since the last term goes to $0$ as $n \to \infty$, we see from \eqref{E:AGF} that $f_{\lambda}^{(j)}(z)\in \TTs\{z/\theta\}$.
\end{proof}

\section{Frobenius modules and rigid analytic trivializations} \label{S:Frobenius}

In this section we determine results on Frobenius modules for Drinfeld $A[\ut_s]$-modules much in line with the theory of dual $t$-motives and pre-$t$-motives associated to abelian $t$-modules (see \cite{ABP04}, \cite{BPrapid}).  In unpublished work Anderson made explicit the connections between periods of abelian $t$-modules and solutions of Frobenius difference equations, such as in \cite{ABP04}, \cite{CP11}, \cite{CP12}, \cite{P08}.  Although unpublished by Anderson, Hartl and Juschka \cite{HartlJuschka16}, \cite{Juschka10}, have written accounts of Anderson's theory.  We follow their exposition but adopt notation that is similar to~\cite{ABP04}, \cite{BPrapid}.  As it is unclear to us yet what the proper theory of dual-$t$-motives should be in this setting, we appeal only to the more general ``Frobenius modules'' (see \cite[\S 2.2]{CPY18}).

\subsection{Frobenius modules}
Similar to $\TTs[\tau]$, the ring $\TTs[\sigma]$ is a noncommutative ring in $\sigma=\tau^{-1}$ with coefficients in $\TTs$ so that $\sigma f=f^{(-1)}\sigma$, for $f\in \TTs$.  Define $* \colon \TTs[\tau] \to \TTs[\sigma]$~by
\[
  f = \sum_i f_i \tau^i \in \TTs[\tau] \quad \mapsto \quad f^{*} := \sum_i f_i^{(-i)} \sigma^i \in \TTs[\sigma],
\]
which satisfies $(f+g)^* = f^* + g^*$ and $(fg)^* = g^* f^*$.  We define a norm $\cdnorm{\,\cdot\,}_{\sigma}$ on $\TTs[\sigma]$ by setting, for an element $f=\sum f_j \sigma^j \in \TTs[\sigma]$, $\cdnorm{f}_{\sigma}:=\sup\{\dnorm{f_j} \mid j \geqslant 0\}$.

\begin{definition}
We now fix a Drinfeld $A[\ut_s]$-module $\phi$ of rank $r$ defined by
\begin{equation} \label{E:Ats}
\phi_{\theta}=\theta + A_1\tau + \dots +A_r\tau^r, \quad A_i \in \TTs,\ A_r \in \TTs^{\times}
\end{equation}
and set $\Lambda_{\phi}=\ker(\exp_{\phi})$. The condition that $A_r \in \TTs^{\times}$ will be crucial to future considerations.  We set $H(\phi):=\TTs[\sigma]$, on which we define a $\TTs[z]$-module structure by setting
\begin{equation}\label{action}
  cz\cdot h=ch\phi_{\theta}^{*}=ch \bigl(\theta + A_1^{(-1)}\sigma +\dots +A_r^{(-r)}\sigma^r \bigr), \quad
  h \in \TTs[\sigma],\ c \in \TTs.
\end{equation}
In this way $H(\phi)$ carries compatible structures of left modules over both $\TTs[\sigma]$ and $\TTs[z]$, and we call $H(\phi)$ the \emph{Frobenius module corresponding to $\phi$}.
\end{definition}

\begin{lemma} \label{L:Hphibasis}
The $\TTs[z]$-module $H(\phi)$ is free and finitely generated with basis $1$, $\sigma, \dots, \sigma^{r-1}$.
\end{lemma}

\begin{proof}
Since $A_r$ is invertible in $\TTs$, we can recursively write every element in $\TTs[\sigma]$ as a $\TTs[z]$-linear combination of $1,\sigma,\dots,\sigma^{r-1}$, using the action defined in \eqref{action}. On the other hand, one can see that the set $\{1,\sigma,\dots,\sigma^{r-1}\}$ is a $\TTs[z]$-basis for $\TTs[\sigma]$ if and only if the set $\{ z^d\cdot \sigma^i \mid d \geqslant 0,\, 0 \leqslant i \leqslant r-1 \}$ is a $\TTs$-basis for $\TTs[\sigma]$.  If we consider again the action in \eqref{action}, we see that $\deg_{\sigma}(z^d\cdot \sigma^i)=rd+i$, for all $d \geqslant 0$ and $0\leqslant i \leqslant r-1$.  Moreover, the leading coefficient of $z^d \cdot \sigma^i$ is in $\TTs^{\times}$, and so the set $\{ z^d \cdot \sigma^i \}$ is a $\TTs$-basis for $\TTs[\sigma]$.
\end{proof}

\subsection{Operators}
Any $f \in \TTs[\tau]$ is necessarily an entire operator and so defines a function $f \colon \TTs \to \TTs$.  We define $f^{\dag} \colon \TTs[\sigma] \to \TTs[\sigma]$ by $f^{\dag}(m) = mf^*$, and we further define $\delta_0, \delta_1 \colon \TTs[\sigma] \to \TTs$ by
\begin{equation} \label{E:delta01}
\delta_0\biggl(\sum_{i\geq 0} a_i \sigma^i \biggr)=a_0, \quad
\delta_1\biggl(\sum_{i\geq 0} a_i \sigma^i \biggr)=\sum_{i\geq 0} a_i^{(i)}.
\end{equation}
We note that $\delta_0$ is a $\TTs$-algebra homomorphism, while $\delta_1$ is $\FF_q[\ut_s]$-linear. We have the following lemma, inspired by the construction of Anderson for $t$-modules over $\CC_{\infty}$ whose proof is found in ~\cite{HartlJuschka16} (see also ~\cite[Lem. 1.1.21-22]{Juschka10}).

\begin{lemma}[{see~Hartl-Juschka \cite[Prop.~5.6]{HartlJuschka16}}] \label{Lemma5-6}
Let $f=\sum_{j=0}^k f_j\tau^j \in \TTs[\tau]$.
\begin{enumerate}
\item[(a)] Define a function $\partial_0f \colon \TTs \to \TTs$ by $\partial_0f(g) = f_0g$.  The following diagram commutes with exact rows:
\begin{displaymath}
\begin{tikzcd}[column sep=large]
0 \arrow{r} &\TTs[\sigma] \arrow{r}{\sigma(\cdot)}\arrow{d}{f^{\dag}}
&\TTs[\sigma]\arrow{r}{\delta_0}\arrow{d}{f^{\dag}}
&\TTs\arrow{d}{\partial_0 f}\arrow{r}&0\\
0 \arrow{r} &\TTs[\sigma]\arrow{r}{\sigma(\cdot)}&\TTs[\sigma]\arrow{r}{\delta_0}&\TTs \arrow{r}&0
\end{tikzcd}
\end{displaymath}
\item[(b)] The following diagram commutes with exact rows:
\begin{displaymath}
\begin{tikzcd}[column sep=large]
0 \arrow{r} &\TTs[\sigma] \arrow{r}{(\sigma-1)(\cdot)}\arrow{d}{f^{\dag}}
&\TTs[\sigma]\arrow{r}{\delta_1}\arrow{d}{f^{\dag}}
&\TTs\arrow{d}{f}\arrow{r}&0\\
0 \arrow{r} &\TTs[\sigma]\arrow{r}{(\sigma-1)(\cdot)}&\TTs[\sigma]\arrow{r}{\delta_1}&\TTs \arrow{r}&0
\end{tikzcd}
\end{displaymath}
In particular,  $\phi_\theta\delta_1=\delta_1\phi_\theta^*$.
\end{enumerate}
\end{lemma}

\subsection{Division towers and exponentiation}

\begin{definition}
For $x\in \TTs$, suppose we have a sequence $\{f_n\}_{n=0}^{\infty}$ in $\TTs$ with
\begin{itemize}
\item $\lim_{n \to \infty} \dnorm{f_n}=0$;
\item $\phi_{\theta}(f_{n+1})=f_n$ for all $n \geqslant 0$;
\item $\phi_{\theta}(f_0)=x$.
\end{itemize}
Such a sequence $\{ f_n \}_{n=0}^{\infty}$ is called a \emph{convergent $\theta$-division tower above $x$}.
\end{definition}

\begin{theorem}[{cf.~Juschka \cite[Prop.~4.1.21]{Juschka10}}]\label{Theorem17}
Let $\phi$ be a Drinfeld $A[\ut_s]$-module as in \eqref{E:Ats}.  Let $x\in \TTs$.  Then there is a canonical bijection
\[
  G \colon \{ \xi \in \TTs \mid \exp_{\phi}(\xi) = x \} \to \{ \textup{convergent $\theta$-division towers above $x$} \},
\]
defined by
\[
  G(\xi) := \biggl\{ \exp_{\phi} \biggl( \frac{\xi}{\theta^{n+1}} \biggr)\biggm| n \geqslant 0 \biggr\}.
\]
Furthermore, if $\{ f_n \}_{n=0}^\infty$ is a convergent $\theta$-division tower above $x$, then with respect to  $\dnorm{\,\cdot\,}$,
\[
  \lim_{n \to \infty} \theta^{n+1} f_n = \xi,
\]
where $\exp_\phi(\xi) = x$ and $G(\xi) = \{ f_n \}$.
\end{theorem}

\begin{remark}
By this theorem, we see that $\phi$ is uniformizable if and only if for any $x\in \TT_s$, there is a convergent $\theta$-division tower above $x$.
\end{remark}

\begin{proof}[{Proof of Theorem \ref{Theorem17}}]
We know from~\eqref{E:exp} that $\phi_{\theta}(\exp_{\phi}(\xi/\theta^{n+1})) = \exp_{\phi}(\xi/\theta^n)$ for all $n\geqslant 0$, and in particular when $n=0$, $\phi_{\theta}\exp_{\phi}(\xi/\theta) = \exp_{\phi}(\xi) = x$.  At the same time $\dnorm{\exp_{\phi}(\xi/\theta^n)} \to 0$.  Therefore, the map $G$ is well-defined.

To show that $G$ is injective, we suppose that $\xi$, $\xi' \in \TTs$ satisfy $\exp_{\phi}(\xi) = \exp_{\phi}(\xi') = x$ with $G(\xi) = G(\xi')$.  This implies that for all $n \geqslant 0$,
\[
  \exp_{\phi} \biggl( \frac{\xi}{\theta^{n+1}} \biggr) = \exp_{\phi} \biggl( \frac{\xi'}{\theta^{n+1}} \biggr),
\]
which implies that $(\xi - \xi')/\theta^{n+1} \in \Lambda_\phi$ for all $n \geqslant 0$. Since $\Lambda_{\phi}$ is discrete, $\xi = \xi'$.

We now show that $G$ is surjective.  Suppose that $\{ f_n \}_{n=0}^{\infty}$ is a convergent $\theta$-division tower above $x$.  By its convergence and Lemma~\ref{L:iso}, we see that there exists $n_0 \geqslant 0$ so that $\log_{\phi}(f_n)$ converges in $\TTs$ for all $n \geqslant n_0$.  We let $\xi := \theta^{n+1} \log_{\phi}(f_n)$ for any $n \geqslant n_0$.  Noting that by the functional equation for~$\log_{\phi}$ and the defining properties of $\{ f_n \}$,
\[
  \theta^{n+2} \log_{\phi}(f_{n+1}) = \theta^{n+1} \log_{\phi}(\phi_{\theta}(f_{n+1})) = \theta^{n+1}\log_{\phi}(f_n), \quad \forall\,n \geqslant n_0,
\]
and so our element $\xi$ does not depend on the choice of $n \geqslant n_0$.  Thus for $n \geqslant n_0$, $f_n = \exp_{\phi} ( \xi/\theta^{n+1})$.  Now for $n < n_0$, we have
\[
  f_n = \phi_{\theta^{n_0-n}}(f_{n_0}) = \phi_{\theta^{n_0-n}} \biggl( \exp_{\phi}\biggl( \frac{\xi}{\theta^{n_0+1}} \biggr)\biggr) = \exp_{\phi} \biggl( \frac{\xi}{\theta^{n+1}} \biggr),
\]
where the last equality follows from~\eqref{E:exp}, and so $G(\xi) = \{ f_n \}_{n=0}^{\infty}$.

Now given a convergent $\theta$-division tower $\{ f_n \}$ above $x$, we let $\xi \in \TTs$ be the unique element such that $G(\xi) = \{ f_n \}$.  Then
\begin{align*}
  \lim_{n \to \infty} \bigl\lVert \xi - \theta^{n+1}f_n \bigr\rVert
    &= \lim_{n \to \infty} \biggl\lVert \xi-\theta^{n+1} \exp_{\phi} \biggl( \frac{\xi}{\theta^{n+1}} \biggr) \biggr\rVert \\
    &= \lim_{n \to \infty} \biggl\lVert \theta^{n+1} \sum_{j\geqslant 1}\alpha_j\frac{\xi^{(j)}}{\theta^{q^jn+q^j}} \biggr\rVert \\
    &= \lim_{n \to \infty} \biggl\lVert  \sum_{j\geqslant 1}\theta^{(1-q^j)(n+1)}\alpha_j\xi^{(j)} \biggr\rVert \\
	&\leqslant \lim_{n \to \infty} \biggl\lVert \theta^{(1-q)(n+1)} \exp_{\phi}(\xi) \biggr\rVert =0,
\end{align*}
which proves the last assertion.
\end{proof}

\subsection{$z$-frames} \label{SS:zframes}
Consider now the $\TTs[z]$-module action on $H(\phi)$ as defined in~\eqref{action}. Let $\bp=[p_1,\dots,p_r]^{\tr}\in \Mat_{r\times 1}(H(\phi))$ be a basis for $H(\phi)$. Let
\[
  \iota \colon \Mat_{1 \times r}(\TTs[z]) \to \TTs[\sigma]
\]
be the map defined for $\bh = [h_1, \dots, h_r] \in \Mat_{1\times r}(\TTs[z])$, by
\begin{equation} \label{E:iotah}
  \iota(\bh) = \bh \cdot \bp = h_1 \cdot p_1 + h_2 \cdot p_2 + \dots + h_r\cdot p_r.
\end{equation}
Let $\Phi \in \Mat_{r}(\TTs[z])$ be the matrix defined by the equation $\sigma \bp=\Phi\bp$. We have the following lemma.
\begin{lemma}\label{Lemma14}
For the map $\iota$ and the matrix $\Phi$, the following holds.
\begin{enumerate}
\item[(a)] $\det(\Phi)= c(z-\theta)$ where $c\in \TTs^{\times}$.
\item[(b)] For all $\bh \in \Mat_{1 \times r}(\TTs[z])$, we have $\iota( \bh^{(-1)}\Phi)=\sigma \iota(\bh)$.
\item[(c)] For all $\bh \in \Mat_{1 \times r}(\TTs[z])$, we have $\iota(z\bh)=\iota(\bh)\phi^{*}_{\theta}.$
\end{enumerate}
\end{lemma}
In the sense of Anderson, a $z$-frame $(\iota,\Phi)$ for $\phi$ is a choice of a basis $\bp$ for $H(\phi)$ satisfying the statements of Lemma \ref{Lemma14}. We now introduce an example of a $z$-frame for $\phi$.

Since $H(\phi) = \TTs[\sigma]$, we can take $\bp := [ 1, \sigma, \dots, \sigma^{r-1} ]^{\tr} \in \Mat_{r\times 1}(H(\phi))$ as a basis of $H(\phi)$. Note that $\sigma \bp=\Phi\bp$, where $\Phi \in \Mat_{r}(\TTs[z])$ can be defined as
\begin{equation} \label{E:Phidef}
\Phi=\begin{bmatrix}
    0 & 1 & \cdots & 0 \\
    \vdots & \vdots & \ddots & \vdots \\
	0 & 0 & \cdots & 1 \\
    \dfrac{(z - \theta)}{A_r^{(-r)}} & -\dfrac{A_1^{(-1)}}{A_r^{(-r)}} & \cdots & -\dfrac{A_{r-1}^{(-r+1)}}{A_r^{(-r)}}    \\
\end{bmatrix}.
\end{equation}

It is easy to show that $(\iota,\Phi)$ is a $z$-frame for $\phi$. Moreover, this particular choice of the $z$-frame  $(\iota,\Phi)$ for $\phi$ will be our main interest throughout the paper.

\begin{remark} \label{R:extension}
If we consider the map $\delta_0 \circ \iota \colon (\Mat_{1 \times r}(\TTs[z]),\cdnorm{\,\cdot\,}_{\theta}) \to (\TTs,\dnorm{\,\cdot\,})$, then for $\bh = [h_1, \dots, h_r] \in \Mat_{1 \times r}(\TTs[z])$, we have
\[
\dnorm{\delta_{0}\circ \iota(\bh)}=\cdnorm{h_1(\theta)}\leqslant \cdnorm{h_1}_{\theta}\leqslant \cdnorm{\bh}_{\theta},
\]
and so the map $\delta_0 \circ \iota$ is bounded. Since $\Mat_{1 \times r}(\TTs[z])$ is $\cdnorm{\,\cdot\,}_{\theta}$-dense in $\Mat_{1 \times r}(\TTs\{ z/\theta \})$, we can extend $\delta_0 \circ \iota$ to a map
\[
D_0 \colon \Mat_{1 \times r}(\TTs\{z/\theta\}) \to \TTs
\]
of complete normed modules, where we recall the definition of $\TTs\{z/\theta\}$ from \S\ref{SS:Tatealgebras}. Furthermore, for $\bg = [g_1, \dots, g_r] \in \Mat_{1\times r}(\TTs\{ z/\theta\})$ and $\bh = [h_1, \dots, h_r] \in \Mat_{1\times r}(\TTs [z])$, it follows from \eqref{E:iotah} that
\begin{equation}
  D_0(\bg+\bh) = g_1(\theta)+h_1(\theta) = g_1|_{z=\theta}
\end{equation}
and
\[
\delta_1 \circ \iota(\bh)=h_1^{(0)}+\dots +h_r^{(r-1)}.
\]
\end{remark}

\begin{theorem}[{cf.~Hartl-Juschka \cite[Thm.~5.18]{HartlJuschka16}}]\label{Theorem18}
Let $\phi$ be a Drinfeld $A[\ut_s]$-module as in~\eqref{E:Ats}, and let $(\iota,\Phi)$ be the $z$-frame for $\phi$ as defined in \eqref{E:Phidef}. Fix $\bh \in \Mat_{1\times r}(\TTs[z])$, and suppose there exists $\bg \in \Mat_{1 \times r}(\TTs \{ z/\theta \})$ satisfying the functional equation
\[
\bg^{(-1)}\Phi -\bg=\bh.
\]
Letting $\Xi = \delta_1(\iota(\bh)) \in \TTs$ and $\xi = D_0(\bg+\bh)$, we have
\[
  \exp_{\phi}(\xi) = \Xi.
\]
\end{theorem}

\begin{proof}
The arguments go back to Anderson, but we follow parts~6 and~7 of the proof of \cite[Thm.~5.18]{HartlJuschka16}.  Let $\bg = \sum_{i=0}^{\infty} \bg_i z^i$ where $\bg_i \in \Mat_{1 \times r}(\TTs)$. For each $n \geqslant 0$, set $\bg_{\leqslant n} := \sum_{i \leqslant n} \bg_i z^i$ and $\bg_{> n} := \sum_{i > n} \bg_i z^i$.  We let
\begin{equation} \label{E:hn}
  \bh_n:=\frac{ \bg_{> n}^{(-1)}\Phi - \bg_{> n}}{z^{n+1}}=\frac{\bh + \bg_{\leqslant n}- \bg_{\leqslant n}^{(-1)}\Phi}{z^{n+1}} \in \Mat_{1 \times r}(\TTs[z]).
\end{equation}
To justify what is claimed to be true in the definition of $\bh_n$, we observe that since $\bg^{(-1)}\Phi - \bg = \bh$, the equality in~\eqref{E:hn} holds. By the definition of $\bg_{> n}$, we see that the first expression is divisible by $z^{n+1}$ and is a power series in $z$. Since $\bg_{\leqslant n}$ is a polynomial in $z$, the second expression implies that each entry of $\bh_n$ must then be a polynomial in~$z$. Furthermore,
$\deg_z (\bh_n) \leqslant \max \{ \deg_z(\bh)-n-1, 0 \} \leqslant \deg_z(\bh)$,
and so the entries of $\bh_n$ are polynomials in $z$ of degree bounded independently of $n$.  Therefore, the $\bh_n$'s live in a free and finitely generated sub-$\TTs$-module $V$ of $\Mat_{1\times r}(\TTs[z])$.

We now show that $\{\delta_1(\iota(\bh_n))\}_{n=0}^{\infty}$ is a convergent $\theta$-division tower above $\delta_1(\iota(\bh))$.  Note
\begin{align}
  \delta_1(\iota(\bh_n))-\phi_{\theta}(\delta_1(\iota(\bh_{n+1}))) &= \delta_1(\iota(\bh_n)) - \delta_1(\iota(\bh_{n+1}) \cdot \phi_{\theta}^{*}) & &\textup{(Lemma~\ref{Lemma5-6}(b)),}\\
  &=\delta_1(\iota(\bh_{n} - z\bh_{n+1})) & &\textup{(Lemma~\ref{Lemma14}(c)).} \notag
\end{align}
On the other hand,
\[
  \bh_n - z \bh_{n+1} = \frac{\bh + \bg_{\leqslant n} - \bg_{\leqslant n}^{(-1)} \Phi - \bh - \bg_{\leqslant n+1} + \bg_{\leqslant n+1}^{(-1)}\Phi}{z^{n+1}} = \biggl(\frac{\bg_{n+1}}{z^{n+1}}\biggr)^{(-1)}\Phi - \frac{\bg_{n+1}}{z^{n+1}}.
\]
Thus,
\begin{equation}
\begin{aligned}
\delta_1(\iota(\bh_{n} - z\bh_{n+1})) &= \delta_1\biggl( \iota\biggl(\biggl( \frac{\bg_{n+1}}{z^{n+1}} \biggr)^{(-1)}\Phi - \frac{\bg_{n+1}}{z^{n+1}} \biggr)\biggr) & & \\
&=\delta_1((\sigma-1)\iota(\bg_{n+1}/z^{n+1})) & & \textup{(Lemma \ref{Lemma14}(b)),} \\
&=0 & & \textup{(Lemma \ref{Lemma5-6}(b)).}
\end{aligned}
\end{equation}
That is, for all $n \geqslant 0$,
\begin{equation}\label{55}
\delta_1(\iota(\bh_n))=\phi_{\theta}(\delta_1(\iota(\bh_{n+1}))).
\end{equation}
A similar calculation shows that $\phi_{\theta}(\iota(\bh_0)) = \delta_1(\iota(\bh))$. We recall the definition of the norm $ \|\cdot\|_1$ from Section 2.2 and the norm $ \|\cdot\|_{\sigma}$ from Section 4.1 and note that since $\dnorm{\bg_n} \to 0$ as $n \to \infty$, we have that $\cdnorm{\bg_{>n}}_1 \to 0$ as $n \to \infty$.  Noting that $\cdnorm{ \bg_{>n}^{(-1)}\Phi}_1 \leqslant \cdnorm{\bg_{>n}}^{1/q}_1\cdnorm{\Phi}_1$, we see that $\cdnorm{\bh_n}_1 \to 0$.  By Lemma~\ref{L:EquivNorms}, the restriction of norms $\cdnorm{\,\cdot\,}_1$ and $\cdnorm{\iota(\,\cdot\,)}_{\sigma}$ on $V$ are equivalent.  Thus
\begin{equation} \label{E:iotahnlimit}
\cdnorm{\iota(\bh_n)}_{\sigma} \to 0.
\end{equation}
Since the degree of $\bh_n$ in $z$ is bounded independently of $n$, the degree of $\iota(\bh_n)$ in $\sigma$ is similarly bounded independently of $n$, say $\deg_\sigma \iota(\bh_n) \leqslant n_0$.  Therefore, for $n$ large enough if we take $\iota(\bh_n)=\sum_{j=0}^{n_0} c_j \sigma^j$, then $\cdnorm{c_j}_{\sigma}\leqslant 1$, and so
\begin{equation} \label{delta1est}
\dnorm{\delta_1(\iota(\bh_n))} = \biggl\lVert \sum_{j=0}^{n_0} c_j^{(j)} \biggr\rVert_{\infty} \leqslant \sup\{\dnorm{c_j} \mid j \geqslant 0\} = \cdnorm{\iota(\bh_n)}_{\sigma}.
\end{equation}
Then \eqref{E:iotahnlimit} implies that $\dnorm{\delta_1(\iota(\bh_n))} \to 0$, and so $\{\delta_1(\iota(\bh_n))\}_{n=0}^{\infty}$ is a convergent $\theta$-division tower above $\Xi = \delta_1(\iota(\bh))$.

Now let $\xi = D_0(\bg + \bh)$.  We claim that with respect to $\dnorm{\,\cdot\,}$,
\begin{equation}
  \lim_{n\to \infty} \theta^{n+1} \delta_1(\iota(\bh_n)) = \xi,
\end{equation}
after which by Theorem~\ref{Theorem17}, $\exp_{\phi}(\xi) = \Xi$, and we are done.  We have
\begin{align*}
\xi &=\lim_{n \to \infty} \delta_0(\iota(\bg_{\leqslant n} + \bh)) & & \textup{(definition of $\xi$),}  \\
  &=\lim_{n \to \infty} \delta_0\bigl(\iota\bigl(z^{n+1} \bh_n +  \bg_{\leqslant n}^{(-1)}\Phi\bigr)\bigr) & & \textup{(by \eqref{E:hn}),} \\
&=\lim_{n \to \infty} \delta_0(\iota(z^{n+1}\bh_n)) & & \textup{(Lemma~\ref{Lemma14}(b)),} \\
&=\lim_{n \to \infty} \delta_0(\iota(\bh_n)\cdot \phi_{\theta^{n+1}}^{*}) & & \textup{(Lemma~\ref{Lemma14}(c)),}   \\
&= \lim_{n \to \infty} \theta^{n+1} \delta_0(\iota(\bh_n)) & & \textup{(Lemma~\ref{Lemma5-6}(a)).}
\end{align*}
It thus suffices to show that in $\TTs$,
\begin{equation} \label{delta1delta0}
\lim_{n \to \infty} \theta^{n+1} (\delta_1(\iota(\bh_n)) - \delta_0(\iota(\bh_n))) = 0.
\end{equation}
Estimating as in \eqref{delta1est}, for $n$ sufficiently large,
\[
  \dnorm{\delta_1 (\iota(\bh_n)) - \delta_0(\iota(\bh_n))} \leqslant \biggl\lVert \sum_{j=1}^{n_0} c_j^{(j)} \biggr\rVert_{\infty}
  \leqslant \sup \bigl\{ \dnorm{c_j}^{q^j} \bigr\} \leqslant \cdnorm{\iota(\bh_n)}_{\sigma}^q,
\]
and so~\eqref{delta1delta0} will follow by showing $\lim_{n\to 0} \bigl\lvert \theta^{n+1} \bigr\rvert_{\infty} \cdot \cdnorm{\iota(\bh_n)}_{\sigma}^q = 0$.  On the other hand, for~$n$ sufficiently large, $\cdnorm{\iota(\bh_n)}_{\sigma} \leqslant 1$, and so it suffices to show that $\lim_{n\to 0} \lvert \theta^{n+1}\rvert_{\infty} \cdot \cdnorm{ \iota(\bh_n)}_{\sigma} = 0$.  Since by Lemma~\ref{L:EquivNorms}, $\cdnorm{\,\cdot\,}_{\theta}$ and $\cdnorm{\iota(\,\cdot\,)}_\sigma$ are equivalent on $V$, it finally suffices to show that
\begin{equation} \label{E:finalred}
  \lim_{n \to 0} \bigl\lvert \theta^{n+1} \bigr\rvert_{\infty} \cdot \cdnorm{ \bh_n }_{\theta} = 0.
\end{equation}
By~\eqref{E:hn}, $\bh_n = (\bg_{>n}^{(-1)} \Phi - \bg_{>n})/z^{n+1} = \sum_{i=0}^{d_0} a_i z^i$, with $a_i \in \Mat_{1\times r}(\TTs)$ and $d_0$ independent of $n$, and so
\begin{align*}
  \bigl\lvert \theta^{n+1} \bigr\rvert_{\infty} \cdot \cdnorm{ \bh_n}_{\theta}
  &= \bigl\lvert \theta^{n+1} \bigr\rvert_{\infty} \left\lVert \frac{\bg_{>n}^{(-1)} \Phi - \bg_{>n}}{z^{n+1}} \right\rVert_{\theta} \\
  &= \bigl\lvert \theta^{n+1} \bigr\rvert_{\infty} \cdot \sup_i \bigl\lvert \theta^{i} \bigr\rvert_{\infty} \cdot \cdnorm{a_i}_1 \\
  &= \sup_i \bigl\lvert \theta^{n+1+i} \bigr\rvert_{\infty} \cdot \cdnorm{a_i}_1 \\
  &= \bigl\lVert a_0 z^{n+1} + a_1 z^{n+2} + \cdots + a_d z^{n+d+1} \bigr\rVert_{\theta}  \\
  &= \bigl\lVert \bg_{>n}^{(-1)} \Phi - \bg_{>n} \bigr\rVert_{\theta}.
\end{align*}
Now since $\bg \in \Mat_{1\times r} (\TTs \{ z/\theta\})$, it follows that $\cdnorm{\bg_{> n}}_{\theta} \to 0$ and $\cdnorm{\bg_{> n}^{(-1)} \Phi}_{\theta} \to 0$ as $n \to \infty$, and thus \eqref{E:finalred} holds.
\end{proof}

\subsection{Rigid analytic trivializations} \label{SS:rat}

\begin{definition}
Let $(\iota, \Phi)$ be the $z$-frame as in \eqref{E:Phidef} for the Drinfeld $A[\ut_s]$-module $\phi$ in~\eqref{E:Ats}.  Let $\Psi \in \GL_r(\TTs \{z/\theta\})$ be a matrix such that
\[
 \Psi^{(-1)} = \Phi \Psi.
\]
We say $(\iota,\Phi,\Psi)$ is a \emph{rigid analytic trivialization} of $\phi$.
\end{definition}

\begin{lemma}\label{L:Newton}
Let $u\in \TTs[z]$. Then there exists $U\in \TTs[z]$ such that $U^{(-1)}-U=u$.
\end{lemma}

\begin{proof}
Let $u=\sum_{i=0}^k \bigl(\sum_{\nu} u_{\nu,i}\ut_s^{\nu} \bigr)z^i \in \TTs[z]$.  Thus $u_{\nu,i} \in \CC_{\infty}$ with, for fixed $i$, $|u_{\nu,i}|_\infty \to 0$ as $|\nu| \to \infty$.  Now let $U = \sum_{i=1}^k \bigl( \sum_{\nu} U_{\nu,i}\ut_s^{\nu} \bigr)z^i \in \CC_{\infty}[[t_1,\dots, t_s]][z]$.  If it were the case that
\begin{equation} \label{E:difference}
U^{(-1)}-U=u,
\end{equation}
then we would need to have
\begin{equation} \label{E:polynomial}
  U_{\nu,i}^{1/q} - U_{\nu,i} = u_{\nu,i}, \quad \forall\, i,\, \forall\,\nu \in \ZZ_{\geqslant 0}^s.
\end{equation}
This equation can be solved in $\CC_\infty$, and so we can solve~\eqref{E:difference} in $\CC_{\infty}[[t_1,\dots, t_s]][z]$.  We claim that we can find a solution $U$ of \eqref{E:difference} that is in $\TTs[z]$.  For fixed $i$, and for fixed $\nu$ large enough so that $\ord_{\infty}(u_{\nu,i}) > 0$, the Newton polygon of the polynomial $X^q - X + u_{\nu,i}$ indicates that there is a solution $U_{\nu,i}$ ($=X^{1/q}$) of \eqref{E:polynomial} such that $\ord_{\infty}(U_{\nu,i}) = \ord_{\infty}(u_{\nu,i})/q$.
If for $\nu$ sufficiently large we pick all $U_{\nu,i}$ in this way, then $U \in \TTs[z]$.
\end{proof}

\begin{theorem}[{cf.~Hartl-Juschka \cite[Thm.~5.28]{HartlJuschka16}}] \label{T:uniformizability}
Let $\phi$ be a Drinfeld $A[\ut_s]$-module defined as in~\eqref{E:Ats}.  If $\phi$ has a rigid analytic trivialization $(\iota,\Phi,\Psi)$, then $\phi$ is uniformizable.
\end{theorem}

\begin{remark} We remark that the result in the above theorem is inspired by Anderson's construction for dual $t$-motives over $\CC_{\infty}$. The generalized version of Anderson's original result is used by Hartl and Juschka in \cite[Thm.~5.28]{HartlJuschka16} to give the characterization of uniformizable dual $t$-motives. We also note that the theorem will be also reinforced later by Theorem \ref{T:characterization}.
\end{remark}

\begin{proof}[{Proof of Theorem \ref{T:uniformizability}}]
Let $h_0 \in \TTs$, and let $\bh=[h_0, 0, \dots, 0] \in \Mat_{1\times r}(\TTs)$.  Using the fact that $\Mat_{1 \times r}(\TTs[z])$ is $\cdnorm{\,\cdot\,}_{\theta}$-dense in $\Mat_{1 \times r}(\TTs\{z/\theta\})$, write
\[
\bh\Psi = \bu + \bv,
\]
where $\bu \in \Mat_{1 \times r}(\TTs[z])$ and $\cdnorm{\bv}_{\theta} <1$.  Because $\cdnorm{\bv^{(n)}}_{\theta} < \cdnorm{\bv}_{\theta}^{q^n}$ for all $n \geqslant 0$, the series $V := \sum_{n=1}^{\infty} \bv^{(n)}$ converges in $\Mat_{1 \times r}(\TTs\{z/\theta\})$.  Moreover, $V^{(-1)} - V=\bv$.  By Lemma~\ref{L:Newton}, we pick $U \in \Mat_{1 \times r}(\TTs[z])$ such that $U^{(-1)} - U=\bu$. Letting $\bg:=(U+V)\Psi^{-1}$,
\begin{multline*}
\bg^{(-1)}\Phi - \bg = (U^{(-1)} + V^{(-1)})(\Psi^{(-1)})^{-1} \Phi - (U + V)\Psi^{-1} \\
= (U^{(-1)}-U+H^{(-1)}-H)\Psi^{-1} = (\bu+\bv)\Psi^{-1} = \bh.
\end{multline*}
Moreover, we have that $\delta_1(\iota(\bh))=h_0$, and so by Theorem~\ref{Theorem18}, $\exp_{\phi}(D_0(\bg+\bh))=\delta_1(\iota(h)) = h_0$.  Since the element $h_0 \in \TTs$ was arbitrary, $\exp_{\phi}$ is surjective.
\end{proof}

\begin{corollary}[{cf.~Hartl-Juschka \cite[Cor.~5.19]{HartlJuschka16}}]\label{C:periods}
For any $\lambda \in \Lambda_{\phi}$, there exists $\bg_{\lambda} \in \Mat_{1 \times r}(\TTs \{ z / \theta \} )$ such that
\[
\bg_{\lambda}^{(-1)}\Phi=\bg_{\lambda}
\quad \text{and} \quad
D_0(\bg_{\lambda})=\lambda.
\]
\end{corollary}

\begin{proof}
Let $\lambda \in \Lambda_\phi$, and let $f_{\lambda}(z) \in \TTsz$ be its Anderson generating function as in~\S\ref{SS:AGF}.  Let $\bg_{\lambda} \in \Mat_{1 \times r}(\TTsz )$ be the vector
\begin{equation} \label{E:glambda}
\bg_{\lambda}=- [f_{\lambda}^{(1)}(z), f_{\lambda}^{(2)}(z), \dots, f_{\lambda}^{(r)}(z)]\cdot \begin{bmatrix}
    A_1 & A_2^{(-1)} & A_3^{(-2)} & \cdots & A_r^{(-r + 1)} \\
    A_2 & A_3^{(-1)}&  A_4^{(-2)} & \reflectbox{$\ddots$} & 0 \\
    \vdots & \vdots  & \reflectbox{$\ddots$} & \reflectbox{$\ddots$} & \vdots \\
	 \vdots & A^{(-1)}_{r}  & 0 & \cdots  & 0\\
    A_r & 0 & 0 & \cdots & 0 \\
\end{bmatrix}.
\end{equation}
By Proposition \ref{P:ResSol}(b), $\bg_{\lambda} \in \Mat_{1\times r}(\TTs\{ z/\theta\})$.  Then by Proposition~\ref{P:ResSol}(c) and some calculation, $\bg_{\lambda}^{(-1)}\Phi = \bg_{\lambda}$.  Finally by Proposition~\ref{P:ResSol}(b) and Remark~\ref{R:extension}, $D_0(\bg_{\lambda}) = \lambda$.
\end{proof}

\begin{proposition} \label{P:Vphi}
Let $\phi$ be a Drinfeld $A[\ut_s]$-module as in~\eqref{E:Ats}.  Suppose that $(\iota, \Phi, \Psi)$ is a rigid analytic trivialization of $\phi$, and set
\begin{equation} \label{E:Vphi}
  V_{\phi} := \{ \bg \in \Mat_{1\times r}(\TTs\{z/\theta\}) \mid \bg^{(-1)}\Phi=\bg \}.
\end{equation}
The following hold.
\begin{enumerate}
\item[(a)] $V_{\phi} = \Mat_{1\times r}(\FF_q[\ut_s][z])\Psi^{-1}$.
\item[(b)] $V_{\phi}$ is a free $\FF_q[\ut_s][z]$-module of rank $r$.
\item[(c)] $V_{\phi} \cap  \Mat_{1\times r}(\TTs[z]) = \{ 0\}$.
\end{enumerate}
\end{proposition}

\begin{proof}
We let $\bg_1, \dots, \bg_r$ be the rows of $\Psi^{-1}$, and since $(\Psi^{-1})^{(-1)} \Phi = \Psi^{-1}$, we see that $\bg_1, \dots, \bg_r \in V_{\phi}$.  We claim that
\begin{equation} \label{E:V}
  V_{\phi}  = \FF_q[\ut_s][z] \cdot \bg_1 + \cdots + \FF_q[\ut_s][z] \cdot \bg_r.
\end{equation}
Certainly the right-hand side is contained in $V_{\phi}$, so let $\bg \in V_{\phi}$ be arbitrary.  Since $\bg_1, \dots, \bg_r$ form a $\TTs\{z/\theta\}$-basis of $\Mat_{1\times r}(\TTs\{ z/\theta\})$, we can find $\beta_1, \dots, \beta_r\in \TTs\{z/\theta\}$ so that $\bg = \beta_1 \bg_1 + \cdots + \beta_r \bg_r$.  As $\bg$, $\bg_1, \dots, \bg_r \in V_{\phi}$,
\[
0 = \bg^{(-1)} \Phi - \bg = \sum_{i=1}^r \beta_i^{(-1)} \bg_i^{(-1)} \Phi - \sum_{i=1}^r \beta_i \bg_i = \sum_{i=1}^r \bigl( \beta_i^{(-1)} - \beta_i \bigr) \bg_i,
\]
and by the linear independence of $\bg_1, \dots, \bg_r$, it follows that for each $i$, $\beta_i^{(-1)} - \beta_i = 0$.  Thus for each $i$, Lemma~\ref{L:fixed} implies that $\beta_i \in \FF_q[\ut_s][z]$, which finishes the claim in~\eqref{E:V}.  Thus (a) holds, and since $\bg_1, \dots, \bg_r$ are $\TTs\{z/\theta\}$-linearly independent, (b) also follows.

For part (c), let $\bg = [g_1, \dots, g_r] \in V_{\phi} \cap \Mat_{1\times r}(\TTs[z])$. Then since $\bg^{(-1)}\Phi-\bg = 0$, substituting in the definition of $\Phi$ from~\eqref{E:Phidef} provides
\begin{equation}\label{injectivity5}
g_1 = \frac{z-\theta}{A_r^{(-r)}}g_r^{(-1)}, \quad
g_2 = g_1^{(-1)}-\frac{A_1^{(-1)}}{A_r^{(-r)}}g_r^{(-1)}, \quad \ldots, \quad
g_r = g_{r-1}^{(-1)}-\frac{A_{r-1}^{(-r+1)}}{A_r^{(-r)}}g_r^{(-1)}.
\end{equation}
Applying $\tau^{j-1}$ to the $j$-th equation in~\eqref{injectivity5} and then writing each $g_i$ in terms of $g_r$, the last equation yields
\begin{equation}\label{injectivity6}
g_r^{(r-1)}=\frac{z-\theta}{A_r^{(-r)}}g_r^{(-1)}-\frac{A_1}{A_r^{(-(r-1))}}g_r -\dots - \frac{A_{r-1}}{A_r^{(-1)}}g_r^{(r-2)}.
\end{equation}
We observe that the left-hand side of~\eqref{injectivity6} is a polynomial in $z$ of degree $\deg_z g_r$ and that the right-hand side is a polynomial in $z$ of degree $\deg_z g_r + 1$. This is a contradiction unless $g_r=0$.  But if $g_r=0$, then~\eqref{injectivity5} implies $g_1 = \dots = g_{r-1}=0$, and thus $\bg=0$.
\end{proof}

\begin{theorem}[{cf.~Hartl-Juschka \cite[Cor.~5.21]{HartlJuschka16}}] \label{T:periods}
Let $\phi$ be a Drinfeld $A[\ut_s]$-module as in~\eqref{E:Ats}.  Suppose $(\iota,\Phi,\Psi)$ is a rigid analytic trivialization of $\phi$, and let $V_{\phi}$ be defined as in~\eqref{E:Vphi}. Then the restriction
\[
D_0|_{V_\phi} \colon V_{\phi} \to \Lambda_\phi
\]
is a bijection, and moreover, $\Lambda_\phi$ is a free $A[\ut_s]$-module of rank~$r$.
\end{theorem}

\begin{proof}
Once we show that $D_0|_{V_{\phi}}$ is a bijection, then it follows from Proposition~\ref{P:Vphi}(b) that $\Lambda_{\phi}$ is a free $A[\ut_s]$-module of rank~$r$.  We note that Corollary~\ref{C:periods} implies
\[
 \Lambda_{\phi} = D_0(V_{\phi}),
\]
and so $D_0|_{V_{\phi}}$ is surjective.  To consider injectivity, let $C$ be the set of convergent $\theta$-division towers above~$0$, and let $G \colon C \to \Lambda_{\phi}$ be the bijection given in Theorem~\ref{Theorem17}.  By the proof of Theorem~\ref{Theorem18}, we know that for any $\bg \in V_\phi$ there exists a convergent $\theta$-division tower $\{\delta_{1}(\iota(\bh_{n}))\}_{n=0}^{\infty}$ above~$0$. Now let
\[
F \colon V_{\phi} \to C
\]
be the map defined by $F(\bg) = \{\delta_{1}(\iota(\bh_{n}))\}_{n=0}^{\infty}$.  Again by Theorem~\ref{Theorem18}, we see that $D_0(\bg)$ is the unique period corresponding to the $\theta$-division sequence $\{\delta_{1}(\iota(\bh_{n}))\}_{n=0}^{\infty}$. Therefore, $D_0|_{V_{\phi}}=G\circ F$. Since $G$ is a bijection, in order to show the injectivity of $D_0|_{V_{\phi}}$, it is enough to prove that $F$ is injective.  Suppose that there exists $\bg \in V_{\phi}$ such that $F(\bg)= \{ \delta_1(\iota(\bh_{n})) \}_{n=0}^{\infty}= ( 0,0,0,\ldots)$.
Since $\ker(\delta_1)=(\sigma-1)\TTs[\sigma]$ and the map $\iota \colon \Mat_{1 \times r}(\TTs[z]) \to \TTs[\sigma]$ is an isomorphism, there exist $\bk_n\in \Mat_{1\times r}(\TTs[z])$ such that for all $n \geqslant 0$, $\iota(\bh_n)=(\sigma-1)\iota(\bk_n)$.
By the definition of~$\bh_n$ in \eqref{E:hn}, we have that
\[
\iota(\bh_n) = \iota\biggl( \frac{\bg_{\leqslant n}- \bg^{(-1)}_{\leqslant n}\Phi}{z^{n+1}} \biggr)=(\sigma-1)\iota(\bk_n) = \iota(\bk_n^{(-1)}\Phi)-\iota(\bk_n),
\]
where the last equality follows from Lemma~\ref{Lemma14}(b).  Since $\iota$ is a $\TTs$-linear map,
\[
 \iota\biggl( \frac{\bg_{\leqslant n}- \bg^{(-1)}_{\leqslant n} \Phi - z^{n+1} \bk_n^{(-1)} \Phi+z^{n+1} \bk_n}{z^{n+1}}\biggr)=0.
\]
Moreover, since $\iota$ is an isomorphism, we find that $\bg_{\leqslant n}- \bg^{(-1)}_{\leqslant n}\Phi-z^{n+1}\bk_n^{(-1)}\Phi +z^{n+1}\bk_n=0$,
and so
\[
 (\bg_{\leqslant n}+z^{n+1}\bk_n)^{(-1)}\Phi= \bg_{\leqslant n}+z^{n+1}\bk_n.
\]
Therefore, $\bg_{\leqslant n}+z^{n+1}\bk_n \in V_{\phi}\cap \Mat_{1\times r}(\TTs[z])$, and so by Proposition~\ref{P:Vphi}(c), for all $n \geqslant 0$,
\[
\bg_{\leqslant n}=-z^{n+1}\bk_n.
\]
Note that the left-hand side is a polynomial in $z$ of degree at most $n$, whereas the right-hand side has degree in $z$ at least $n+1$, unless $\bg_{\leqslant n}=\bk_n=0$.  Therefore, $\bh_n=0$ for all $n\geqslant 0$, and so by Theorem~\ref{Theorem17}, we have $\bg = 0$.  Thus $F$ is injective.
\end{proof}

\section{The de Rham isomorphism} \label{S:deRhamiso}

The theory of biderivations and quasi-periodic extensions of Drinfeld modules was originally explored by Anderson, Deligne, Gekeler, and Yu (see~\cite{Brownawell93}, \cite{Brownawell96}, \cite{BP02}, \cite{Gekeler89}, \cite{Gekeler90}, \cite{Gekeler11}, \cite{Goss94}, \cite{PR03}, \cite{Yu90}, for various treatments).  Our focus in this section is to prove the de Rham isomorphism for Drinfeld $A[\ut_s]$-modules with invertible leading coefficient.  The de Rham isomorphism for constant Drinfeld modules was proved by Gekeler~\cite[Thm.~5.14]{Gekeler89} using quasi-periodic functions, whereas Anderson gave a different proof using rigid analytic trivializations and Anderson generating functions (see Goss~\cite[{\S 1.5}]{Goss94}).  Our treatment follows a hybrid argument, since the analytic arguments of Gekeler over $\CC_{\infty}$ do not completely transfer to the theory of entire operators over~$\TTs$.

\subsection{Biderivations} \label{SS:biderivations}
We fix a Drinfeld $A[\ut_s]$-module $\phi$ of rank $r$ as in~\eqref{E:Ats}, defined by $\phi_\theta=\theta + A_1\tau + \dots + A_r\tau^r$, with $A_r \in \TTs^{\times}$.

\begin{definition}
A \emph{biderivation} is a map $\eta\colon A[\ut_s] \to \tau \TTs[\tau]$ such that
\begin{enumerate}
\item[(a)] $\eta$ is an $\FF_q[\ut_s]$-linear homomorphism;
\item[(b)] $\eta_{ab}=a\eta_b + \eta_a \phi_{b}$ for all $a$, $b\in A[\ut_s]$.
\end{enumerate}
\end{definition}

For fixed $m\in \TTs[\tau]$ and for all $a\in A[\ut_s]$, we say a biderivation $\eta^{\{m\}}$ defined by
\begin{equation}\label{inner}
\eta^{\{m\}}_{a} = m\phi_{a} - am
\end{equation}
is an \emph{inner biderivation}.  As an example, taking $\delta^0 = \eta^{\{ 1 \}}$, we have $\delta^0_a = \phi_a - a$ for all $a \in A[\ut_s]$.  If furthermore $m \in \tau \TTs[\tau]$, then we say $\eta^{\{ m \}}$ is a \emph{strictly inner biderivation}.  Note that \eqref{inner} implies that for any inner biderivation $\eta^{\{ m \}}$,
\[
\deg_{\tau}\eta^{\{m\}}_{\theta}=\deg_{\tau}(m \phi_{\theta})=r+\deg_{\tau}(m)\geqslant r.
\]
Moreover, if $\eta^{\{ m \}}$ is strictly inner, then $\deg_{\tau}\eta^{\{m\}}_{\theta} > r$.  We let $\Der(\phi)$ be the set of all biderivations for $\phi$, and we let
\[
\Der_{si}(\phi) \subseteq \Der_{in}(\phi) \subseteq \Der(\phi)
\]
denote the subsets of all strictly inner and inner biderivations.  Each of these sets possesses the structure of a left $\TTs$-module, by setting for $\eta \in \Der(\phi)$ and $f \in \TTs$ that $(f \cdot \eta)_{\theta} := f\eta_{\theta}$.
Finally, we let $H_{\DR}^*(\phi)$ be the \emph{de Rham module} for $\phi$, which is the $\TTs$-module
\begin{equation}
  H_{\DR}^*(\phi) := \Der(\phi) / \Der_{si}(\phi).
\end{equation}

\begin{lemma}\label{L:valueofti}
Let $\eta \in \Der(\phi)$.  For $1\leqslant i \leqslant s$, we have $\eta_{t_i} = 0$.  Moreover, $\eta$ is uniquely determined by the value $\eta_{\theta}$, and the map
\[
  I \colon \Der(\phi) \to \tau \TTs[\tau],
\]
defined by $I(\eta) = \eta_{\theta}$, is an isomorphism of left $\TTs$-modules.
\end{lemma}

\begin{proof}
Fix $i$ with $1 \leqslant i \leqslant s$.  Since $\theta t_i=t_i\theta$, the definition of biderivation yields that $\eta_{\theta t_i}=\theta \eta_{t_i} + \eta_{\theta} \phi_{t_i} = \theta \eta_{t_i} + \eta_{\theta} t_i$
and $\eta_{t_i\theta}=t_i \eta_{\theta} + \eta_{t_i} \phi_{\theta}$ are equal.  Since each $t_i$ is in the center of $\TTs[\tau]$, we obtain that
\begin{equation}\label{bider1}
\theta \eta_{t_i} = \eta_{t_i} \phi_{\theta}.
\end{equation}
Taking the degree with respect to $\tau$, we have $\deg_{\tau} \eta_{t_i} = \deg_{\tau} \eta_{t_i} + r$, which implies $\eta_{t_i}=0$.

It is straightforward to check that $I$ is a left $\TTs$-module homomorphism.  By the product formula for $\eta$, we see that it is uniquely determined by its values on $t_1, \dots, t_s$ and $\theta$, and since $\eta_{t_i}=0$ for all $i$, $\eta$ is thus determined solely by $\eta_{\theta}$.  Thus $I$ is injective.  Now for any $m \in \tau \TTs[\tau]$, we construct $\eta \in \Der(\phi)$ with $I(\eta) = m$.  We set $\eta_{\theta} := m$, and then define recursively $\eta_{\theta^{j+1}} := \theta \eta_{\theta^{j}} + m \phi_{\theta^j}$.  By routine argument we can extend $\eta$ to a well-defined biderivation $\eta \colon A[\ut_s] \to \tau \TTs[\tau]$.
\end{proof}

\begin{lemma} \label{L:free}
For any $\eta \in \Der(\phi)$ there exist unique $\eta^{*} \in \Der(\phi)$ and $m\in \tau \TTs[\tau]$ such that $\eta = \eta^{*} + \eta^{\{ m \}}$ and $\deg_{\tau} \eta^{*}_{\theta} \leqslant r$.
\end{lemma}

\begin{proof}
We first show existence.  We fix $\eta^{*} \in \eta + \Der_{si}(\phi)$ such that $\deg_{\tau} \eta^{*}_{\theta}$ is minimal.  It suffices to show that $\deg_{\tau} \eta^{*}_{\theta} \leqslant r$.  Suppose instead that $\eta^{*}_{\theta} = c_1\tau + \dots + c_{r+s} \tau^{r+s}$ for $s \geqslant 1$ and $c_{r+s} \neq 0$.  Then letting $m_1 = (c_{r+s}/A_r^{(s)})\tau^s $,
\[
\eta^{\{m_1\}}_{\theta} = m_1 \phi_{\theta} - \theta m_1 = -\frac{\theta c_{r+s}}{A_r^{(s)}}\tau^s + \dots + c_{r+s}\tau^{r+s},
\]
and so if we take $\eta' = \eta^{*} - \eta^{\{m_1\}} \in \eta + \Der_{si}(\phi)$, then the degree in $\tau$ of $\eta'_{\theta} = \eta^{*}_{\theta}-\eta_{\theta}^{\{m_1\}}$ in $\tau$ is strictly less than $\deg_{\tau} \eta^{*}_{\theta} = r+s$.  By the minimality of $\deg_{\tau} \eta^{*}_{\theta}$, we must have $\deg_{\tau} \eta^{*}_{\theta} \leqslant r$.

For uniqueness, let $\eta^{*}_1$, $\eta^{*}_2 \in \Der(\phi)$ and $m_1$, $m_2 \in \tau \TTs[\tau]$ satisfy the conclusions of the lemma.  Then we have that
\[
(\eta_{1}^{*})_{\theta}-(\eta_{2}^{*})_{\theta} = \eta^{\{ m_1-m_2\}}_{\theta}.
\]
If $m_1\neq m_2$, then the degree in $\tau$ of the left-hand side $\leqslant r$, whereas the degree in $\tau$ of the  right-hand side is $>r$. This is a contradiction, and so $\eta_1^{*}=\eta_2^{*}$ and $m_1=m_2$.
\end{proof}

\begin{corollary}\label{free2}
The de Rham module $H_{\DR}^*(\phi) = \Der(\phi)/\Der_{si}(\phi)$ is a free $\TTs$-module of rank $r$ with basis elements represented by $\delta^0$, $\delta^1 ,\dots ,\delta^{r-1} \in \Der(\phi)$, where $\delta^0 = \eta^{\{1\}}$ and for $1 \leqslant i \leqslant r-1$, $\delta^i$ is determined by $\delta^i_{\theta} = \tau^i$.
\end{corollary}

\begin{proof}
Suppose that there exist $b_i \in \TTs$ such that $\sum_{i=0}^{r-1} b_i \delta^{i} = \eta^{\{m\}} \in \Der_{si}(\phi)$ (and so $m \in \tau \TTs[\tau]$).  Evaluating both sides at $\theta$, we have
\[
  b_0 \eta^{\{1\}}_{\theta} + b_1 \tau + \cdots + b_{r-1}\tau^{r-1} = \eta^{\{m\}}_{\theta}.
\]
The degree in $\tau$ of the left-hand side is $\leqslant r$, while the degree in $\tau$ of the right is $>r$.  The only way for this to occur is if $b_i=0$ for all $i$ and $m=0$.  Thus $\delta^{0}$, $\delta^{1}, \dots, \delta^{r-1}$ represent $\TTs$-linearly independent classes in $H_{\DR}^*(\phi)$.

Now let $\eta \in \Der(\phi)$.  By Lemma~\ref{L:free}, the class of $\eta$ in $H_{\DR}^*(\phi)$ is represented by unique $\eta^* \in \Der(\phi)$ with $\eta_{\theta}^{*} = \sum_{i=1}^{r} c_i\tau^i$.  Since $A_r \in \TTs^{\times}$, we also have
\[
\eta_{\theta}^{*}=\sum_{i=1}^r c_i\tau^i =
-\frac{c_r}{A_r}\delta^{0}_{\theta} + \biggl(c_1 - \frac{c_r A_1}{A_r}\biggr)\delta^{1}_{\theta}
+ \dots + \biggl(c_{r-1} - \frac{c_r A_{r-1}}{A_r}\biggr)\delta^{r-1}_{\theta}.
\]
Thus the classes of $\delta^0$, $\delta^1, \dots, \delta^{r-1}$ span $H_{\DR}^*(\phi)$ over $\TTs$, and the result follows.
\end{proof}

\subsection{Quasi-periodic operators}
Let $\eta \in \Der(\phi)$, with $\eta_\theta = \sum_{j=1}^{\ell} c_j \tau^j$.  We claim that there is a unique series $F_{\eta} = \sum_{i=1}^{\infty} f_i\tau^i \in \tau\TTs[[\tau]]$ that satisfies the equation
\begin{equation}\label{E:quasiperiodicfunc}
F_\eta \theta-\theta F_\eta=\eta_{\theta}\exp_{\phi}.
\end{equation}
Recalling that $\exp_{\phi}=\sum_{i=0}^{\infty} \alpha_i \tau^i$, if we compare the coefficients of $\tau^i$ on both sides of~\eqref{E:quasiperiodicfunc}, then we see that for all $i\geqslant 1$, we would require
\begin{equation} \label{E:quasipercoeffs}
f_i =  \frac{1}{\theta^{q^i}-\theta} \sum_{j=1}^{i} c_j\alpha^{(j)}_{i-j},
\end{equation}
where we utilize the convention that $\alpha_{i-j} = 0$ if $i-j < 0$.  This sequence of coefficients $\{ f_i \}$ is uniquely determined by $\eta_\theta$ and $\exp_{\phi}$, and so~\eqref{E:quasiperiodicfunc} has a unique solution.  We call $F_{\eta}$ the \emph{quasi-periodic operator associated to $\eta$}.  Since $\lim_{i\to \infty} \Ord(\alpha_i) = \infty$, it follows that
\[
\lim_{i\to \infty} q^{-i} \Ord(f_i) \geqslant \lim_{i\to \infty} q^{-i} \cdot q^i \cdot \min_{j=1}^{i} \bigl( \Ord(c_j\alpha^{(j)}_{i-j}) \bigr) = \infty.
\]
That is, $F_{\eta}$ is an entire operator and so induces a continuous function $F_{\eta} \colon \TTs \to \TTs$.

As an example, the quasi periodic operator $F_{\delta^{0}}$ corresponding to the biderivation $\delta^{0}$ is
\begin{equation}\label{deltazero}
F_{\delta^{0}} = \exp_{\phi} - 1.
\end{equation}
Furthermore, if $\eta_{\theta}=\sum_{i=0}^{r-1} a_i\delta^i_{\theta}$, then for any $f\in \TTs$, we have
\begin{equation}\label{combination1}
F_{\eta}(f)=a_0F_{\delta^0}(f) +a_1F_{\delta^{1}}(f) + \dots + a_{r-1}F_{\delta^{r-1}}(f).
\end{equation}

\begin{proposition} \label{P:Fetafneq}
Let $F_{\eta}$ be the quasi-periodic operator corresponding to $\eta \in \Der(\phi)$. Then for all $a \in A[\ut_s]$,
we have $F_\eta a-a F_\eta=\eta_{a}\exp_{\phi}$.
\end{proposition}

\begin{proof}
By a straightforward induction on $j\geqslant 1$, using~\eqref{E:quasiperiodicfunc}, we find
\begin{equation}\label{E:linearity2}
F_\eta \theta^j-\theta^j F_\eta=\eta_{\theta^j}\exp_{\phi}.
\end{equation}
Furthermore, we know from Lemma~\ref{L:valueofti} that $\eta_{t_i}=0$ for $1\leqslant i \leqslant s$. Therefore, for any monomial $v\in \FF_q[\ut_s]$, $\eta_{v}=0$. Thus, for $j\geqslant 1$, we have
\begin{equation}\label{linearity3}
\eta_{v\theta^{j}}\exp_{\phi}=v\eta_{\theta^j}\exp_{\phi}=vF_{\eta}\theta^{j} - v\theta^{j}F_{\eta}=F_{\eta}v\theta^{j} - v\theta^{j}F_{\eta}.
\end{equation}
Finally, the result follows from \eqref{E:linearity2} and \eqref{linearity3} for any $a\in A[\ut_s]$.
\end{proof}

\subsection{The de Rham map}
We define the \emph{de Rham map}
\[
\DR\colon H_{\DR}^*(\phi) \to \Hom_{A[\ut_s]}(\Lambda_{\phi},\TTs)
\]
by $\DR([\eta])=F_{\eta }|_{\Lambda_{\phi}}$ where $[\eta]$ is an equivalence class in $H_{\DR}^*(\phi) = \Der(\phi)/\Der_{si}(\phi)$  and $F_{\eta}$ is the quasi-periodic operator associated to $\eta$.

\begin{lemma} \label{L:DRwelldefined}
The map $\DR$ is well-defined and $\TTs$-linear.
\end{lemma}

\begin{proof}
We first show the map is well-defined.  Observe from Proposition~\ref{P:Fetafneq} that for any $[\eta]\in H_{\DR}^*(\phi)$, $a \in A[\ut_s]$, and $\lambda \in \Lambda_{\phi}$, we have $F_{\eta}(a\lambda) - a F_{\eta}(\lambda)=\eta_{a}(\exp_{\phi}(\lambda)) = 0$.  Therefore, $F_{\eta}(a\lambda)=a F_{\eta}(\lambda)$, which implies that the map
\[
F_{\eta}|_{\Lambda_{\phi}} \colon \Lambda_{\phi} \to \TTs
\]
is $A[\ut_s]$-linear. Now assume that $[\eta_1]=[\eta_2]$. Then there exists $m\in \tau\TTs[\tau]$ such that $\eta_1-\eta_2=\eta^{\{m\}}$.  Using \eqref{E:exp}, one shows that $F_{\eta^{\{m\}}} = m\exp_{\phi}$, and thus for any $\lambda \in \Lambda_{\phi}$,
\[
F_{\eta_1-\eta_2}(\lambda) = F_{\eta_1}(\lambda)-F_{\eta_2}(\lambda) = F_{\eta^{\{m\}}}(\lambda) = m(\exp_{\phi}(\lambda)) = 0.
\]
Therefore, $\DR([\eta_1])=\DR([\eta_2])$. Now for $[\eta] \in H_{\DR}^*(\phi)$ and $a\in A[\ut_s]$, \eqref{combination1} implies
\[
\DR(a[\eta])=\DR([a\eta])=F_{a\eta}|_{\Lambda_{\phi}}=aF_{\eta}|_{\Lambda_{\phi}}=a\DR([\eta])
\]
which proves the $A[\ut_s]$-linearity of $\DR$.
\end{proof}

\begin{proposition}[{cf.~Gekeler~\cite[Rmk.~2.7(iii)]{Gekeler89}, Pellarin~\cite[\S 4.2]{Pellarin08}}] \label{P:AndGenFunc}
Let $\lambda$ be an element of $\Lambda_{\phi}$, and let $f_{\lambda}(z) \in \TTsz$ be its associated Anderson generating function from~\S\ref{SS:AGF}.
\begin{enumerate}
\item [(a)] $F_{\delta^0}(\lambda)=\Res_{z=\theta}f_{\lambda}(z)=-\lambda$.
\item [(b)] For any $1\leqslant j \leqslant r-1$, we have
\[
F_{\delta^j}(\lambda)=\sum_{n\geqslant 0} \exp_{\phi} \biggl( \frac{\lambda}{\theta^{n+1}} \biggr)^{(j)} \theta^{n}=f^{(j)}_{\lambda}(z)\big|_{z=\theta}.
\]
\end{enumerate}
\end{proposition}

\begin{proof} Part (a) follows from \eqref{deltazero} and Proposition \ref{P:ResSol}(b). To prove part (b), note that for all $n\geqslant 0$, \eqref{E:quasiperiodicfunc} and the definition of $\delta^{j}$ imply
\[
\theta^nF_{\delta^j} \biggl( \frac{\lambda}{\theta^n} \biggr) = \theta^{n+1} F_{\delta^j} \biggl( \frac{\lambda}{\theta^{n+1}} \biggr) + \theta^{n}\exp_{\phi} \biggl( \frac{\lambda}{\theta^{n+1}} \biggr)^{(j)}.
\]
Therefore, by resubstituting these expressions for $n$ increasing,
\[
F_{\delta^j}(\lambda) = \theta F_{\delta^j} \biggl( \frac{\lambda}{\theta} \biggr) + \exp_{\phi} \biggl( \frac{\lambda}{\theta} \biggr)^{(j)}
= \exp_{\phi} \biggl( \frac{\lambda}{\theta} \biggr)^{(j)} + \theta^2 F_{\delta^j} \biggl( \frac{\lambda}{\theta^2} \biggr) + \theta\exp_{\phi} \biggl( \frac{\lambda}{\theta^2} \biggr)^{(j)},
\]
and moreover for any $N \geqslant 0$,
\[
  F_{\delta^j}(\lambda) = \theta^{N+1} F_{\delta^j} \biggl( \frac{\lambda}{\theta^{N+1}} \biggr) + \sum_{n=0}^{N} \exp_{\phi} \biggl( \frac{\lambda}{\theta^{n+1}} \biggr)^{(j)} \theta^n.
\]
As $N \to \infty$, we have $\theta^{N+1} F_{\delta^j}( \lambda/\theta^{N+1}) \to 0$ in $\TTs$, and so taking $N \to \infty$ we are done.
\end{proof}

\begin{theorem}\label{T:derhamisomorphism}
Let $\phi$ be a Drinfeld $A[\ut_s]$-module defined by $\phi_\theta=\theta + A_1\tau + \dots + A_r\tau^r$,
such that \textup{(i)} $A_r\in \TTs^{\times}$, and \textup{(ii)} $\Lambda_{\phi}$ is a free and finitely generated $A[\ut_s]$-module of rank $r$.  We define the $\TTs$-linear de Rham map
\[
\DR \colon H_{\DR}^*(\phi) \to \Hom_{A[\ut_s]}(\Lambda_{\phi},\TTs)
\]
by $\DR([\eta])=F_{\eta}|_{\Lambda_{\phi}}$. Then $\DR$ is an isomorphism of left $\TTs$-modules.
\end{theorem}

\begin{remark}
We note that Theorem~\ref{T:periods} asserts that if $\phi$ is rigid analytically trivial, then $\Lambda_\phi$ will be free of rank~$r$ over $A[\ut_s]$, and we will verify the converse in Corollary~\ref{C:rigidanalyticityrank}.
\end{remark}

We explain the idea of the proof of Theorem \ref{T:derhamisomorphism}, which is an adaptation of the method of Anderson (see \cite[\S 1.5]{Goss94} for details).  Observe that since by hypothesis $\Lambda_{\phi}$ is free and finitely generated over $A[\ut_s]$ of rank $r$, then $\Hom_{A[\ut_s]}(\Lambda_{\phi},\TTs)$ is free and finitely generated over $\TTs$ of rank~$r$.  Fixing an $A[\ut_s]$-basis $\{ \lambda_1, \dots, \lambda_r \}$ for $\Lambda_\phi$ gives rise to the dual $\TTs$-basis $\{ \Omega_1, \dots, \Omega_r \}$ for $\Hom_{A[\ut_s]}(\Lambda_{\phi},\TTs)$.  Thus we see that
\begin{equation}\label{E:periodmatrix}
\Pi := \begin{bmatrix}
    -\lambda_1 & f^{(1)}_{\lambda_1}(z)|_{z=\theta} & \dots & f^{(r-1)}_{\lambda_1}(z)|_{z=\theta} \\
    -\lambda_2 & f^{(1)}_{\lambda_2}(z)|_{z=\theta} & \dots & f^{(r-1)}_{\lambda_2}(z)|_{z=\theta} \\
    \vdots & \vdots & & \vdots  \\
    -\lambda_r & f^{(1)}_{\lambda_{r}}(z)|_{z=\theta} & \dots & f^{(r-1)}_{\lambda_{r}}(z)|_{z=\theta}\\
\end{bmatrix}
\in \Mat_{r}(\TTs)
\end{equation}
is the matrix representing $\DR$ with respect to the $\TTs$-bases $\{[\delta^0], \dots, [\delta^{r-1}]\}$ and $\{\Omega_1, \dots, \Omega_r\}$.  Now define the matrix
\begin{equation} \label{E:Upsilondef}
\Upsilon := \begin{bmatrix}
    f_{\lambda_1}(z) & f^{(1)}_{\lambda_1}(z) & \dots & f^{(r-1)}_{\lambda_1}(z) \\
    f_{\lambda_2}(z) & f^{(1)}_{\lambda_2}(z) & \dots & f^{(r-1)}_{\lambda_2}(z) \\
    \vdots & \vdots & & \vdots  \\
    f_{\lambda_{r}}(z) & f^{(1)}_{\lambda_{r}}(z) & \dots & f^{(r-1)}_{\lambda_{r}}(z)\\
\end{bmatrix}
\in \Mat_{r}(\TTsz).
\end{equation}
Since $\TTs$ is a commutative ring, $\DR$ is an isomorphism if and only if the matrix $\Pi$ is in $\GL_r(\TTs)$. By Proposition~\ref{P:AndGenFunc}, we see that
\begin{equation} \label{E:PiRes}
  \det(\Pi) = \Res_{z=\theta}\det(\Upsilon),
\end{equation}
and so to prove Theorem~\ref{T:derhamisomorphism} it is then enough to show that $\Res_{z=\theta}\det(\Upsilon) \in \TTs^{\times}$.

\section{Proof of the de Rham isomorphism} \label{S:deRhamisoProof}

This section provides the proof of Theorem~\ref{T:derhamisomorphism}.  As in the previous section, we fix a Drinfeld $A[\ut_s]$-module $\phi$ of rank $r$ defined by $\phi_\theta=\theta + A_1\tau + \dots + A_r\tau^r$ such that $A_r \in \TTs^{\times}$ and $\Lambda_{\phi}$ is a free and finitely generated $A[\ut_s]$-module of rank $r$ with basis $\lambda_1, \dots, \lambda_r$.

\subsection{The matrices $\Upsilon$ and $\Theta$} \label{SS:UpsilonTheta}
We recall the matrices $\Phi \in \Mat_r(\TTs[z])$ from \eqref{E:Phidef} and $\Upsilon \in \Mat_r(\TTsz)$ from \eqref{E:Upsilondef} associated to $\phi$, and following a similar construction in \cite[\S~3.4]{CP12}, we let
\[
V=\begin{bmatrix}
    A_1 & A_2^{(-1)} & A_3^{(-2)} & \cdots & A_r^{(-r + 1)} \\
    A_2 & A_3^{(-1)}&  A_4^{(-2)} & \reflectbox{$\ddots$} & 0 \\
    \vdots & \vdots  & \reflectbox{$\ddots$} & \reflectbox{$\ddots$} & \vdots \\
	 \vdots & A^{(-1)}_{r}  & 0 & \cdots  & 0\\
    A_r & 0 & 0 & \cdots & 0 \\
\end{bmatrix} \in \Mat_{r}(\TTs).
\]
Now by setting $\Theta = \Upsilon^{(1)}V$, Proposition~\ref{P:ResSol}(c) implies that
\begin{equation}\label{E:Thetafunc}
 \Theta^{(-1)}\Phi=\Theta.
\end{equation}
We note that under the condition that $\Theta$ is invertible, the matrix $\Psi = \Theta^{-1}$ is a natural candidate to be a rigid analytic trivialization for $\phi$; we investigate this possibility in Corollary~\ref{C:rigidanalyticityrank}.  Recall now from \S\ref{SS:ATelements} that for any $\alpha\in \TTsz^{\times}$, we have the Anderson-Thakur element $\omega(\alpha)\in \TTsz^{\times}$.

\begin{lemma}\label{L:unit}
For some $\gamma \in \FF_q[\ut_s,z]$,
\[
\det(\Theta) = \gamma\omega \left( \frac{(-1)^{r-1}A_r^{(-r+1)}}{(z-\theta^q)} \right)^{-1}.
\]
\end{lemma}

\begin{proof}
Note that applying $\tau$ to both sides of \eqref{E:Thetafunc} and using Lemma \ref{Lemma14}(a) yields
\begin{equation}\label{Thetaequation}
\det(\Theta) \cdot (-1)^{r-1}\frac{(z-\theta^q)}{A_r^{(-r+1)}}=\tau(\det(\Theta)).
\end{equation}
Letting $\Delta = A_r^{(-r+1)}/(z-\theta^q) \cdot \tau - 1$,
then \eqref{Thetaequation} shows that $\det(\Theta) \in \Sol_{s,z}(\Delta,\TTsz)$. Since $A_r^{(-r+1)}/(z-\theta^q) \in \TTsz^{\times}$, the result follows from Proposition~\ref{P:taudiff}.
\end{proof}

Now we need a lemma to show that $\det(\Theta)$ is nonzero.

\begin{lemma}\label{L:indepen}
For any $k \geqslant 0$, the functions $f_{\lambda_1}^{(k)}(z), \dots, f_{\lambda_r}^{(k)}(z)$ are linearly independent over $\FF_q(\ut_s,z)$.
\end{lemma}

\begin{proof}
For the case $k=0$, we adapt the ideas of the proof of \cite[Lem.~3.4.4]{CP12}.  Assume to the contrary that there exists $c_1(z),\dots,c_r(z)$ in $\FF_q(\ut_s,z)$, not all of zero, so that
\[
c_1(z)f_{\lambda_1}(z) + \dots + c_r(z)f_{\lambda_r}(z)=0.
\]
Then by Proposition~\ref{P:ResSol}(b), we have that
\[
\Res_{z=\theta} \bigl( c_1(z)f_{\lambda_1}(z) + \dots+ c_r(z)f_{\lambda_r}(z) \bigr)
= -c_1(\theta)\lambda_1 - \dots - c_r(\theta)\lambda_r=0.
\]
But we know that $\lambda_1, \dots, \lambda_r$ are linearly independent over $\FF_q(\theta,\ut_s)$, and so $c_1(z) = \dots = c_r(z) = 0$, which contradicts the choice of $c_i$'s. For $k \geqslant 1$, assume that there exists $d_1(z),\dots,d_r(z)$ in $\FF_q(\ut_s,z)$, not all of zero, so that $d_1(z)f^{(k)}_{\lambda_1}(z) + \dots + d_r(z)f^{(k)}_{\lambda_r}(z)=0$.  Since the elements of $\FF_q(\ut_s,z)$ are fixed under twisting, we have that $(d_1(z)f_{\lambda_1}(z) + \dots + d_r(z)f_{\lambda_r}(z))^{(k)}=0$, and so the result follows from the $k=0$ case.
\end{proof}

Note that $\TTsz$ is a difference algebra with the infinite order automorphism $\tau$ (see \cite{Cohn}, \cite{vdPutSinger97} and \cite{vdPutSinger03} for details about difference algebras). Recall that by Lemma \ref{L:fixed}, the polynomial ring $\mathbb{F}_q[\ut_s,z]$ is the set of elements of $\TTsz$ which are fixed by the automorphism $\tau$. Using Lemma \ref{L:indepen}, next proposition will be a standard fact in difference algebras.

\begin{proposition}[see {Cohn \cite[Chap.~8, Lem.~II]{Cohn}}] \label{P:nonvanish}
The determinants $\det(\Upsilon)$ and $\det(\Theta)$ are nonzero.
\end{proposition}

\subsection{Invertibility of $\Theta$ over $\TTsz$}

\begin{lemma}\label{AndTha}
Let $\Delta=A_r\tau^r + \dots + A_1\tau - (z -\theta) \in \TTsz[\tau]$. Then $\Sol_{s,z}(\Delta,\TTsz)$ is a free and finitely generated $\FF_q[\ut_s,z]$-module of rank $r$ with basis $f_{\lambda_1}(z), \dots , f_{\lambda_r}(z)$.
\end{lemma}

\begin{proof}
By Lemma~\ref{L:indepen}, we know that $f_{\lambda_1}(z), \dots, f_{\lambda_r}(z)$ are $\FF_q[\ut_s,z]$-linearly independent.  On the other hand, by Proposition~\ref{P:ResSol}(c), $\Delta(f_{\lambda_i}(z))=0$ for $1\leqslant i \leqslant r$.  Now let $Y=\sum_{i=0}^{\infty} a_i z^i$, $a_i \in \TTs$, be in $\Sol_{s,z}(\Delta,\TTsz)$. This implies that
\[
\sum_{i=0}^{\infty} \bigl(\theta a_i + A_1 a^{(1)}_i + \dots + A_ra_i^{(r)} \bigr) z^i=
\sum_{i=0}^{\infty} a_i z^{i+1}.
\]
If we compare coefficients of $z^i$ on both sides, we see that $\phi_{\theta}(a_i)=a_{i-1}$ for all $i \geqslant 1$ and $\phi_{\theta}(a_0)=0$. Since $Y\in \TTsz$, we know that $\dnorm{a_i} \to 0$ as $i \to \infty$. Therefore, the sequence $ \{ a_i \}_{i=0}^{\infty}$ is a convergent $\theta$-division tower above $0$, and so by Theorem~\ref{Theorem17}, there exists unique $\lambda \in \Lambda_{\phi} $ such that $\exp_{\phi}( \lambda/\theta^{i+1}) = a_i$ for all $i \geqslant 0$. Thus if $\lambda = c_1\lambda_1 + \cdots + c_r \lambda_r$, for $c_i \in A[\ut_s]$, then $Y= c_1|_{\theta=z} \cdot f_{\lambda_1}(z) + \dots +  c_r|_{\theta=z}\cdot f_{\lambda_r}(z)$.
\end{proof}

Our goal now is to show that $\det(\Theta) \in \TTsz^{\times}$.  Let $\bp \in \Mat_{r \times 1}(\TTs[\sigma])$ be the column vector $\bp = [1, \sigma, \dots, \sigma^{r-1}]^{\tr}$ used in defining the $z$-frame of $\phi$ in \S\ref{SS:zframes}.  We also let $\iota \colon \Mat_{1\times r}(\TTs[z]) \to \TTs[\sigma]$ be the associated isomorphism of $\TTs[z]$-modules defined in~\eqref{E:iotah}.
We let
\[
\bP := \TTsz \otimes_{\TTs[z]} H(\phi),
\]
which is a free $\TTsz$-module, and we define the $\sigma$-action diagonally on the elements of $\bP$.  In what follows we identify $H(\phi)$ with its image $1 \otimes H(\phi) \subseteq \bP$.  Finally, letting
\[
\bP^{\sigma} := \{ \mu \in \bP \mid \sigma \mu = \mu \}
\]
be the set of $\sigma$-invariant elements of $\bP$, we note that $\bP^{\sigma}$ is an $\FF_q[\ut_s,z]$-module.

\begin{lemma}\label{L:inv}
The $\FF_q[\ut_s,z]$-module $\bP^{\sigma}$ is free and finitely generated of rank~$r$, and the entries of the column vector $\Theta \bp$ form a basis.
\end{lemma}

\begin{proof}
By Proposition~\ref{P:nonvanish}, we know that $\det(\Theta) \neq 0$, and so the entries of $\Theta \bp$ are $\FF_q[\ut_s,z]$-linearly independent. On the other hand, by~\eqref{E:Phidef} and~\eqref{E:Thetafunc},
\[
\sigma(\Theta \bp) = \Theta^{(-1)}\sigma(\bp)=\Theta^{(-1)} \Phi \bp=\Theta \bp.
\]
Therefore, each entry of $\Theta \bp$ is an element of $\bP^{\sigma}$.  Moreover, for any $Q \in \Mat_{1 \times r}(\TTsz)$ such that $Q\bp \in \bP^{\sigma}$, we have that
$
\sigma(Q \bp)=Q^{(-1)}\sigma(\bp)=Q^{(-1)} \Phi \bp=Q \bp,
$
and so
$ Q=Q^{(-1)}\Phi.$
Letting $Q=[Q_1, \dots, Q_r]$, this implies
\begin{equation}\label{E:basis}
Q_1=\frac{z-\theta}{A_r^{(-r)}}Q_r^{(-1)}, \quad
Q_2=Q_1^{(-1)}-\frac{A_1^{(-1)}}{A_r^{(-r)}}Q_r^{(-1)}, \quad
\ldots, \quad
Q_r =Q_{r-1}^{(-1)}-\frac{A_{r-1}^{(-r+1)}}{A_r^{(-r)}}Q_r^{(-1)}.
\end{equation}
If we apply $\tau^{j-1}$ to the $j$-th equation in \eqref{E:basis} and then write each $Q_i$ in terms of $Q_r$, we find from the last equation that
\[
Q_r^{(r-1)} = A_r\biggl(\frac{Q_r^{(-1)}}{A_r^{(-r)}}\biggr)^{(r)} = (z-\theta) \frac{Q_r^{(-1)}}{A_r^{(-r)}}-A_1 \biggl(\frac{Q_r^{(-1)}}{A_r^{(-r)}}\biggr)^{(1)} - \cdots - A_{r-1} \biggl( \frac{Q_r^{(-1)}}{A_r^{(-r)}} \biggr)^{(r-1)}.
\]
In other words, $Q_r^{(-1)}/A_r^{(-r)}$ is an element of $\Sol_{s,z}(\Delta,\TTsz)$.  Thus by Lemma~\ref{AndTha}, for some $c_i \in \FF_q[\ut_s,z]$, we have $Q_r^{(-1)}=A_r^{(-r)}\sum_{i=1}^r c_if_{\lambda_i}(z)$
and therefore $Q_r=A_r^{(-r+1)}\sum_{i=1}^r  c_if^{(1)}_{\lambda_i}(z)$. Similarly, using the first equation above and Proposition~\ref{P:ResSol}(c), we find that
\[
Q_1=\frac{z-\theta}{A_r^{(-r)}}Q_r^{(-1)}= \sum_{i=1}^r  c_i(z-\theta)f_{\lambda_i}(z)=\sum_{i=1}^r  c_i(A_1f^{(1)}_{\lambda_i}(z)+\dots+A_rf^{(r)}_{\lambda_i}(z)).
\]
Continuing in this manner using \eqref{E:basis}, we conclude that
$
Q\bp=[c_1,c_2,\ldots,c_r]\Theta \bp.
$
Therefore, every element in $\bP^{\sigma}$ is in the $\FF_q[\ut_s,z]$-linear span of the entries of $\Theta \bp$.
\end{proof}

\begin{proposition}\label{P:unit2}
The determinant $\det(\Theta)$ is in $\TTsz^{\times}$.
\end{proposition}

\begin{proof}
Let $\gamma \in \FF_q[\ut_s,z]$ be as in Lemma~\ref{L:unit}, and moreover $\gamma \neq 0$ by Proposition~\ref{P:nonvanish}.  It then suffices to show that $\gamma \in \FF_q^{\times}$.  To do this we modify ideas in \cite[Prop.~3.3.9]{P08}, and assume to the contrary that $\gamma \in \FF_q[\ut_s,z] \setminus \FF_q$.  By Lemma \ref{L:unit}, we have that $\det(\Theta) \equiv 0 \pmod{\gamma}$.  Now, there exists a nonzero $\bff =[f_1,\dots, f_r] \in \Mat_{1 \times r}(\TTsz)$ such that $\bff\Theta \equiv 0 \pmod \gamma$.
Dividing $\bff$ by a suitable element in $\CC_{\infty}$, without loss of generality, we can assume that $\dnorm{f_i}\leqslant 1$ for all $i$ and for at least one $i$, $\dnorm{f_i}=1$.

For any given $h\in \TTsz$, let us denote $h=\sum_{ \underline{\nu} \in \NN^{s+1}} h_{ \underline{\nu}}(\ut_s z)^{ \underline{\nu}}$ where $(\ut_s z)^{\underline{\nu}}:=t_1^{\nu_1}\dots t_s^{\nu_s}z^{\nu_{s+1}}$ and $h_{ \underline{\nu}}\in \CC_{\infty}$ with $\inorm{h_{ \underline{\nu}}} \to 0$ as $\nu_1+\dots +\nu_s+\nu_{s+1} \to \infty$. For any $h\in \TTsz$ with $\dnorm{h}\leqslant 1$, there exist only finitely many multi-indices $ \underline{\nu}_1, \dots,  \underline{\nu}_m$ whose corresponding coefficients have norm $1$.  Now, fix a lexicographic order with respect to $t_1, \dots, t_s,z$ on $\FF_q[\ut_s,z]$.  Let $ \underline{\nu}_j$ be the multi-index among $ \underline{\nu}_1, \dots,  \underline{\nu}_m$ such that its corresponding monomial is greatest with respect to the lexicographic order.  Then we can write $h = v + g$ such that $v:=h_{ \underline{\nu}_1}(\ut_s z)^{ \underline{\nu}_1}+\dots + h_{ \underline{\nu}_m}(\ut_s z)^{ \underline{\nu}_m} \in \CC_{\infty}[\ut_s,z]$ and $g:=\sum_{\nu \not\in \{ \underline{\nu}_1,\dots, \underline{\nu}_m\}}h_{ \underline{\nu}}(\ut_s z)^{ \underline{\nu}} \in \TTsz$ satisfy the following properties: (i) the monomial corresponding to $ \underline{\nu}_j$ in $v$ is its leading monomial with respect to the lexicographic order, and (ii) $\dnorm{g} < 1$.  By the multivariable division algorithm \cite[Chap.~2, Thm.~3]{CoxLittleOShea}, there exist $v_{h}$, $r_{h} \in \CC_{\infty}[\ut_s,z]$ such that $v = v_{h}\gamma + r_{h}$ and that none of the monomials of $r_h$ are divisible by the leading term of $\gamma$.  Thus we have $h=v_{h}\gamma + r_{h} +g \equiv r_{h} + g  \pmod{\gamma}$.

Therefore, without loss of generality, we can further assume that for all $i$, $f_i=r_i + g_i $, where (i) $r_i \in \CC_{\infty}[\ut_s,z]$ satisfies that none of its monomials are divisible by the leading term of $\gamma$, and (ii) $g_i\in \TTsz$ with $\dnorm{g_i} < 1$.  Now by \eqref{E:Thetafunc} we have
\begin{equation}\label{finallemma1}
\bff\Theta^{(-1)}\Phi = \bff\Theta\equiv 0 \pmod \gamma.
\end{equation}
Since by Lemma \ref{Lemma14}(a), $\det(\Phi)$ is invertible in $\TTsz$, \eqref{finallemma1} implies that $\bff\Theta^{(-1)} \equiv 0 \pmod \gamma$. Moreover, for any $n\geqslant 2$, we have by induction that
\begin{equation}\label{finallemma2}
\bff\Theta^{(-n)}\equiv \bff\Theta^{(-n)}\Phi^{(-(n-1))}=\bff\Theta^{(-(n-1))} \equiv 0 \pmod \gamma.
\end{equation}
Since $\gamma$ is invariant under twisting, \eqref{finallemma2} implies that
\begin{equation}\label{finallemma3}
\bff^{(n)}\Theta=(\bff\Theta^{(-n)})^{(n)} \equiv 0 \pmod \gamma.
\end{equation}
Define $\bv := [v_1, \dots, v_n]:=\Theta \bp$ and observe that by Lemma \ref{L:inv}, the entries of $\bv$ form $\FF_q[\ut_s,z]$-basis for $\bP^{\sigma}$. Since $\bff\Theta \equiv 0 \pmod \gamma$, we have that
\[
\frac{1}{\gamma}\bff\bv=\frac{1}{\gamma}\bff\Theta \bp \in \bP.
\]
Furthermore, \eqref{finallemma3} implies for all $n\geqslant 0$,
\[
\frac{\bff^{(n)}}{\gamma}\Theta \bp=\frac{\bff^{(n)}}{\gamma}\bv \in \bP.
\]
Now, define a norm $\cdnorm{\, \cdot \,}_{\bP}$ on $\bP$ by $
\bigl\lVert\sum h_i\sigma^i\bigr\rVert_{\bP}=\sup \dnorm{h_i}$
where $h_i \in \TTsz$. Since $\TTsz$ is complete with respect to $\dnorm{\, \cdot \,}$, $\cdnorm{\, \cdot \,}_{\bP}$ is a complete norm on $\bP$, and for all $g\in \TTsz$ and $\beta \in \bP$, we have $\cdnorm{g\beta}_{\bP} = \dnorm{g}\cdnorm{\beta}_{\bP}$.  By Lemma \ref{L:limit}, there exists $m>0$ such that with respect to the $\cdnorm{\, \cdot \,}_{\bP}$-metric,
\begin{equation}\label{polynomials}
\lim_{n \to \infty} \frac{1}{\gamma}\sum f_i^{(mn)}v_i= \lim_{n \to \infty} \frac{1}{\gamma}\sum (r_i+g_i)^{(mn)}v_i=   \lim_{n \to \infty} \frac{1}{\gamma}\sum r_i^{(mn)}v_i=\frac{1}{\gamma}\sum c_iv_i \in \bP.
\end{equation}
where $c_i \in \overline{\FF}_q[\ut_s,z]$ with $c_j \neq 0$.  (We have used the fact that $\dnorm{g_i}<1$ for the second equality in \eqref{polynomials}.) For some $l\geqslant 1$, we have each $c_{i} \in  \FF_{q^l}[\ut_s,z]$, and we set $d_i:=c_i +c_i^{(-1)}+ \dots + c_i^{(1-l)}$.  Since the image of the trace map $\Tr \colon \FF_{q^{l}} \to \FF_q$ is non-trivial, we can divide $c_j$ by a suitable element of $\FF_{q^l}^{\times}$ and assume that $d_j \neq 0$. Using the fact that $\sigma(\bv)=\bv$, we have
\[
\mu := \sum_{k=0}^{l-1}\sigma^{k} \biggl( \frac{1}{\gamma}\sum_{i=1}^{r} c_i v_i\biggr) = \frac{1}{\gamma}\sum_{i=1}^{r} \biggl(\sum_{k=0}^{l-1} c_i^{(-k)} \biggr) v_i = \frac{1}{\gamma}\sum_{i=1}^r d_i v_i\in \bP.
\]
Since $d_i \in \FF_{q}[\ut_s,z]$, $\mu$ is invariant under $\sigma$, and so $\mu \in \bP^{\sigma}$. Now since none of the monomials of $r_i$ are divisible by the leading term of $\gamma$ for $1\leqslant i \leqslant r$, it follows that $\gamma$ does not divide $d_j$. This contradicts the fact that by the construction of $\bv$, its entries comprise an $\FF_q[\ut_s,z]$-basis for $\bP^{\sigma}$. Thus $\gamma \in \FF_q^{\times}$, and therefore $\det(\Theta) \in \TTsz^{\times} .$
\end{proof}

Fixing a $(q-1)$-st root $(-\theta)^{1/(q-1)}$ of $-\theta$, we define the Carlitz period $\tilde{\pi}\in \CC_{\infty}^{\times}$ by
\[
\tpi = \theta(-\theta)^{1/(q-1)}\prod_{i=1}^{\infty} \Bigl(1- \theta^{1-q^i}\Bigr)^{-1}
\]
and define $\Omega(z) \in \CC_{\infty}\{ z /\theta\}$ by
\begin{equation}\label{E:Omega}
\Omega(z)=(-\theta)^{-q/(q-1)}\prod_{i=1}^{\infty} \biggl(1-\frac{z}{\theta^{q^i}}\biggr).
\end{equation}
By choosing $x=-1/\theta^q$ in~\eqref{E:omega}, we see that $\omega(1/(z-\theta^q)) = -\Omega(z)$.

\begin{corollary}\label{C:rigidanalyticityrank}
For a Drinfeld $A[\ut_s]$-module $\phi$ defined by
$\phi_\theta=\theta + A_1\tau + \dots + A_r\tau^r$ with $A_r \in \TTs^{\times}$, $\Lambda_{\phi}$ is free of rank $r$ over $A[\ut_s]$ if and only if $\phi$ is rigid analytically trivial.
\end{corollary}

\begin{proof}
Suppose that $\Lambda_{\phi}$ is free of rank $r$ over $A[\ut_s]$. Let $\Theta \in \Mat_r(\TTsz)$ be defined as in \S\ref{SS:UpsilonTheta}. By Proposition~\ref{P:ResSol}(b), $\Theta \in \Mat_r(\TTs\{z/\theta\})$. By \eqref{E:property}, Lemma \ref{L:unit}, and Proposition~\ref{P:unit2}, it follows that for some $d\in \FF_q^{\times}$,
\[
\det(\Theta) = d\omega((-1)^{r-1}A_r^{(-r+1)})^{-1}\omega(1/(z-\theta^q))^{-1} = d\omega((-1)^{r-1}A_r^{(-r+1)})^{-1}\Omega(z)^{-1}.
\]
Since $\Omega(z)^{-1}$ has poles at $z=\theta^{q^n}$ for $n\geqslant 1$, we find $\Theta \in \GL_r(\TTs\{z/\theta\})$. Moreover by \eqref{E:Thetafunc}, $(\Theta^{-1})^{(-1)}=\Phi\Theta^{-1}$.  Therefore, $(\iota,\Phi,\Theta^{-1})$ is a rigid analytic trivialization of $\phi$. The opposite direction follows from Theorem \ref{T:periods}.
\end{proof}

\subsection{Completion of the proof}
Observe that $\Omega(z)$ satisfies the identities $\Omega(\theta) = -\tilde{\pi}^{-1}$ and $\Omega^{(-1)}(z)=(z-\theta)\Omega(z)$.  Thus,
\begin{equation}\label{Omega3}
\Res_{z=\theta}\frac{1}{\omega(1/(z-\theta^q))^{(-1)}}=\Res_{z=\theta} \frac{-1}{\Omega(z)^{(-1)}} = \tpi.
\end{equation}

\begin{proof}[Proof of Theorem~\ref{T:derhamisomorphism}]
Lemma~\ref{L:unit} and Proposition~\ref{P:unit2} imply that for some $c \in \FF_q^{\times}$,
\begin{equation}\label{determinant}
\det(\Theta) = \det(\Upsilon)^{(1)} (-1)^{r-1} \prod_{i=0}^{r-1} A_r^{(-i)} = c\omega \biggl( \frac{(-1)^{r-1}A_r^{(-r+1)}}{(z-\theta^q)} \biggr)^{-1}.
\end{equation}
Since $A_r\in \TTs^{\times}$ by assumption, \eqref{determinant} implies that $\det(\Upsilon)\in \TTsz^{\times}$.  Using \eqref{E:property}, \eqref{Omega3}, and \eqref{determinant}, we have for some $d \in \FF_q^{\times}$,
\[
\Res_{z=\theta}\det(\Upsilon) = d \tpi \biggl( \omega\bigl( (-1)^{r-1}A_r^{(-r+1)} \bigr)^{(-1)} \prod_{i=1}^r A_r^{(-i)} \biggr)^{-1} \in \TTs^{\times},
\]
which completes the proof by~\eqref{E:PiRes}.
\end{proof}

\section{Uniformizability Criteria} \label{S:Uniformizability}

\subsection{Uniformizability of Drinfeld $A[ \protect \ut_s]$-modules}
Let $\phi$ be an $A[\ut_s]$-module of arbitrary rank $r \geqslant 1$, and set $\Lambda_{\phi}:=\Ker(\exp_{\phi})$.  We set $\phi[\theta] := \{f \in \TTs \mid \phi_{\theta}(f)=0\}$.

\begin{theorem} \label{T:characterization}
Let $\phi$ be a Drinfeld $A[\ut_s]$-module of rank $r$ defined by
\[
\phi_\theta=\theta + A_1\tau + \dots + A_r\tau^r
\]
such that \textup{(i)} $A_r \in \TTs^{\times}$, and \textup{(ii)} $\Lambda_{\phi}$ is a free and finitely generated $A[\ut_s]$-module.  Then the following are equivalent.
\begin{enumerate}
\item[(a)] $\Lambda_{\phi}$ is free of rank $r$ over $A[\ut_s]$.
\item[(b)] $\phi$ has a rigid analytic trivialization.
\item[(c)] The de Rham map $\DR$ is an isomorphism.
\item[(d)] $\phi$ is uniformizable, and $\phi[\theta]$ is free of rank $r$ over $\FF_q[\ut_s]$.
\end{enumerate}
\end{theorem}

\begin{proof}
Note that (a) $\Leftrightarrow$ (b) follows from Corollary~\ref{C:rigidanalyticityrank}. We first prove (a) $\Leftrightarrow$ (c). Observe that Theorem~\ref{T:uniformizability} together with Corollary~\ref{C:rigidanalyticityrank} yields (a) $\Rightarrow$ (c).  On the other hand, if $\DR$ is an isomorphism, then by Corollary \ref{free2}, we have that $H_{\DR}^*(\phi)\cong \TTs^{\oplus r}\cong\Hom_{A[\ut_s]}(\Lambda_{\phi},\TTs)$. But by the assumption, $\Lambda_{\phi}\cong A[\ut_s]^{\oplus x}$ for some $x\in \NN$, we see that
\[
\TTs^{\oplus r}\cong\Hom_{A[\ut_s]}(\Lambda_{\phi},\TTs)\cong\Hom_{A[\ut_s]}(A[\ut_s]^{\oplus x},\TTs)\cong \TTs^{\oplus x} .
\]
Thus $x=r$, which proves (c) $\Rightarrow$ (a).  Now we prove (a) $\Rightarrow$ (d).  If $\Lambda_{\phi}$ is free of rank $r$ over $A[\ut_s]$, then Theorem~\ref{T:uniformizability} and Corollary~\ref{C:rigidanalyticityrank} imply that $\phi$ is uniformizable. Moreover, uniformizability implies that $
\phi[\theta] \cong \Lambda_{\phi}/\theta \Lambda_{\phi} \cong \FF_q[\ut_s]^{\oplus r}.$
Finally, we prove (d) $\Rightarrow$ (a).  By uniformizability, we have $\Lambda_{\phi}/\theta\Lambda_{\phi}\cong \phi[\theta]$.  Therefore,
\[
\rank_{A[\ut_s]}\Lambda_{\phi} = \rank_{\FF_q[\ut_s]} \Lambda_{\phi}/\theta\Lambda_{\phi} =\rank_{\FF_q[\ut_s]} \phi[\theta] =r.
\]
Thus $\Lambda_{\phi}$ is free of rank $r$ over $A[\ut_s]$.
\end{proof}

\section{Applications and examples} \label{S:Applications}

\subsection{Analogue of the Legendre relation in $\TTs$.}
Let $\phi$ be a Drinfeld $A[\ut_s]$-module of rank $r$ defined by $\phi_{\theta}=\theta + A_1\tau + \dots + A_r \tau^r$ such that $A_r \in \TTs^{\times}$ and $\Lambda_{\phi}$ is a free $A[\ut_s]$-module with basis elements $\lambda_1, \dots, \lambda_r$.  Using \eqref{E:Thetafunc}, Lemma~\ref{L:unit}, and Proposition~\ref{P:unit2}, we find
\[
\det(\Upsilon) = c \Bigl( A_r^{(-r)} A_r^{(r-1)}\cdots A_r^{(-1)} \omega((-1)^{r-1}A_r^{(-r+1)}/(z-\theta^q))^{(-1)} \Bigr)^{-1},
\]
for some $c\in \FF_q^{\times}$. On the other hand, observe that
\begin{equation}\label{A2}
A_r^{(-r)}A_r^{(r-1)}\dots A_r^{(-1)} \omega((-1)^{r-1}A_r^{(-r+1)})^{(-1)}=\omega((-1)^{r-1}A_r).
\end{equation}
Finally, using \eqref{E:property}, \eqref{Omega3}, and \eqref{A2}, for some $d \in \FF_q^{\times}$, we have
\begin{equation}
\Res_{z=\theta}\det(\Upsilon) = \det(\Pi)= d \tpi \bigl( \omega((-1)^{r-1}A_r) \bigr)^{-1},
\end{equation}
where $\Pi$ is the matrix defined as in \eqref{E:periodmatrix}.

\begin{remark} When $r=2$, we recover the usual Legendre relation: by Proposition~\ref{P:AndGenFunc}(b),
\[
\Res_{z=\theta} \det(\Upsilon) = \det(\Pi)= \lambda_2 F_{\delta^1}(\lambda_1) - \lambda_1 F_{\delta^1}(\lambda_2) = \frac{d \tpi}{\omega(-A_2)}
\]
for some $d\in \FF_q^{\times}$, which can be seen as the analogue of the Legendre relation in $\TTs$.
\end{remark}

\subsection{Constant Drinfeld $A[\ut_s]$-modules}
Suppose that we have a constant Drinfeld $A[\ut_s]$-module $\phi$, i.e., $\phi_{\theta} = \theta + A_1 \tau + \dots + A_r \tau^r$ with each $A_i \in \CC_{\infty}$ and $A_r \neq 0$.  Naturally we can restrict $\phi$ to $A$ and obtain a traditional Drinfeld $A$-module over $\CC_{\infty}$, and we can ask to what extent do the fundamental properties of traditional Drinfeld modules, in terms of uniformizability and period lattices, translate into properties of the constant Drinfeld $A[\ut_s]$-module~$\phi$.  The answers are satisfactory.

\begin{proposition} \label{P:constant}
Let $\phi$ be a Drinfeld $A[\ut_s]$-module of rank $r\geqslant 1$ which is isomorphic to a constant Drinfeld $A[\ut_s]$-module.  Then the equivalent statements of Theorem~\ref{T:characterization} all hold.  In particular, $\phi$ is uniformizable, and its period lattice $\Lambda_{\phi}$ is free of rank $r$ over $A[\ut_s]$.
\end{proposition}

\begin{proof}
If $\phi$ is isomorphic to the constant Drinfeld $A[\ut_s]$-module $D$, then there exists $u \in \TTs^{\times}$ so that for all $a \in A[\ut_s]$, $\phi_a = uD_a u^{-1}$.  Moreover, the uniqueness of the exponential function and \eqref{E:exp} readily imply that $\exp_\phi = u \exp_D u^{-1}$.  Since the properties in Theorem~\ref{T:characterization} are invariant under isomorphism, it suffices to assume that $\phi$ itself is constant.

In this case, by \eqref{E:exp} we see that $\exp_{\phi} \in \CC_{\infty} [[\tau]]$ (in fact we obtain the same series while considering $\phi$ as a Drinfeld $A[\ut_s]$-module over $\TTs$ or as a Drinfeld $A$-module over~$\CC_{\infty}$).  By fundamental theory of Drinfeld $A$-modules (see \cite[Ch.~4]{Goss}), the induced function $\exp_\phi \colon \CC_{\infty} \to \CC_{\infty}$ is surjective, and its kernel $\widetilde{\Lambda}_{\phi}$ is rank~$r$ over $A$, say with basis $\lambda_1, \dots, \lambda_r \in \CC_{\infty}$.

Now for $f = \sum_{\nu} a_\nu \ut_s^\nu \in \TTs$, we see that
\begin{equation} \label{E:expconstant}
  \exp_{\phi}(f) = \sum_{\nu} \exp_{\phi}(a_\nu) \ut_s^{\nu},
\end{equation}
and since $\exp_{\phi}(a_\nu) \in \CC_{\infty}$, we see that the surjectivity of $\exp_{\phi} \colon \TTs\to \TTs$ follows from the surjectivity of $\exp_{\phi} \colon \CC_{\infty} \to \CC_{\infty}$.  Thus $\phi$ is uniformizable.

Clearly, $\Lambda_{\phi} \supseteq \Span_{A[\ut_s]} (\lambda_1, \dots, \lambda_r)$, and we prove the reverse containment.  Suppose that $\lambda = \sum_{\nu} \ell_{\nu} \ut_s^{\nu} \in \Lambda_{\phi}$.  Then by~\eqref{E:expconstant}, for each $\nu$, $\exp_{\phi}(\ell_\nu) = 0$, and so $\ell_{\nu} \in \widetilde{\Lambda}_{\phi} = \Span_A(\lambda_1, \dots, \lambda_r)$.  However, since $\lambda \in \TTs$, we must have $|\ell_{\nu}|_{\infty} \to 0$ as $|\nu| \to \infty$, and since $\widetilde{\Lambda}_{\phi} \subseteq \CC_{\infty}$ is discrete, we see that for $|\nu|$ sufficiently large, $\ell_{\nu} = 0$.  Thus $\lambda \in \Span_{A[\ut_s]} (\lambda_1, \dots, \lambda_r)$.  Therefore, $\Lambda_{\phi} = \Span_{A[\ut_s]} (\lambda_1, \dots, \lambda_r)$.  Furthermore,  as $A[\ut_s]$-modules $\Span_{A[\ut_s]} (\lambda_1, \dots, \lambda_r) \cong A[\ut_s] \otimes_{A} \widetilde{\Lambda}_{\phi}$, and so $\Lambda_\phi$ is free of rank $r$ over $A[\ut_s]$.
\end{proof}

\begin{corollary} \label{C:constant}
Suppose that $\phi$ is a constant Drinfeld $A[\ut_s]$-module of rank $r \geqslant 1$, and suppose that $\lambda_1, \dots, \lambda_r \in \CC_{\infty}$ form an $A$-basis of the period lattice of $\phi$, when considered as a Drinfeld $A$-module over $\CC_{\infty}$.  Then $\lambda_1, \dots, \lambda_r$ form an $A[\ut_s]$-basis for $\Lambda_{\phi} \subseteq \TTs$.
\end{corollary}

\subsection{Examples of non-isotrivial uniformizable Drinfeld $A[\ut_s]$-modules}
As we saw in Proposition~\ref{P:char} (due to Angl\`{e}s, Pellarin, and Tavares Ribeiro~\cite[Prop.~6.2]{AnglesPellarinTavares16}), for rank~$1$ Drinfeld $A[\ut_s]$-modules, uniformizability is equivalent to being isomorphic to the Carlitz module over $\TTs$.  In light of Proposition~\ref{P:constant} and Corollary~\ref{C:constant}, in this section we investigate examples of uniformizable Drinfeld $A[\ut_s]$-modules of rank $r \geqslant 2$ that are not isomorphic to constant Drinfeld modules.

We will say that a Drinfeld $A[\ut_s]$-module is \emph{isotrivial} if it is isomorphic to a constant Drinfeld module.  The following theorem provides a way to construct uniformizable non-isotrivial Drinfeld $A[\ut_s]$-modules.  We recall from Lemma~\ref{L:iso} that there exists $\varepsilon_{\phi} > 0$ such that $\exp_{\phi}$ is an isometric automorphism with its inverse $\log_{\phi}$ on the set $\{f \in \TTs \mid \dnorm{f} < \varepsilon_{\phi}\}$.

\begin{theorem} \label{T:RANKROC}
Let $\phi$ be a Drinfeld $A[\ut_s]$-module of rank $r \geqslant 1$.  Suppose that for some $m \geqslant 1$, $\phi[\theta]$ is free of rank $m$ over $\FF_q[\ut_s]$ with basis elements $\gamma_1, \dots, \gamma_m$.  Suppose further that for $1 \leqslant i \leqslant m$, we have $\dnorm{\gamma_i} < \varepsilon_{\phi}$.  Then $\Lambda_{\phi}$ is free of rank $m$ over $A[\ut_s]$, and $\theta \log_{\phi}(\gamma_1), \dots, \theta \log_{\phi}(\gamma_m)$ form an $A[\ut_s]$-basis.
\end{theorem}

\begin{proof}
For $1\leqslant i \leqslant m$, set $\lambda_i := \theta \log_{\phi}(\gamma_i)$.  We claim that $\lambda_1, \dots, \lambda_m$ are $A[\ut_s]$-linearly independent.  Suppose that there exist $a_i = a_{i,0} + a_{i,1}\theta + \dots + a_{i,k_i}\theta^{k_i}\in A[\ut_s]$, where $a_{i,j}\in \FF_q[\ut_s]$, such that $a_1 \lambda_1 + \cdots + a_m \lambda_m = 0$.  For all $j\geqslant 0$, by \eqref{E:exp}, we have
\[
\exp_{\phi}(\theta^{j}\lambda_i) = \phi_{\theta^{j+1}} \biggl(\exp_{\phi} \biggl( \frac{\lambda_i}{\theta} \biggr)\biggr) = \phi_{\theta^{j+1}} (\gamma_i)=\phi_{\theta^{j}}(\phi_{\theta}(\gamma_i)) = 0.
\]
Therefore, considering $\exp_\phi(a_1 \lambda_1/\theta + \cdots + a_m\lambda_m/\theta) = 0$, we find $a_{1,0}\gamma_1 + \cdots + a_{m,0}\gamma_m =0$.  Since $\gamma_1, \dots, \gamma_m$ are $\FF_q[\ut_s]$-linearly independent, it follows that  $a_{i,0}=0$ for all $i$.  Inductively, we find that $a_{i,k}=0$ for all $k\geqslant 0$ and $1\leqslant i \leqslant m$. Thus the claim follows.

Now we show that $\lambda_1, \dots, \lambda_m$ generate $\Lambda_{\phi}$ as an $A[\ut_s]$-module.  Let $X_0 \in \Lambda_{\phi}$.  If $X_0=0$, then we are done.  If not, the discreteness of $\Lambda_{\phi}$ allows us to pick $n_0 \geqslant 1$ such that $\exp_{\phi}(X_0/\theta^{n_0-1})= 0$ and $\exp_{\phi}(X_0/\theta^{n_0})\neq 0$.  Since $\exp_{\phi}(X_0/\theta^{n_0}) \in \phi[\theta]$,
\begin{equation}\label{E:x0kernel}
\exp_{\phi} \biggl( \frac{X_0}{\theta^{n_0}} \biggr) = \sum b_{0,i} \exp_{\phi}\biggl( \frac{\lambda_i}{\theta} \biggr)
\end{equation}
for some $b_{0,i}\in \FF_q[\ut_s]$.  Collecting all of the terms of~\eqref{E:x0kernel} to the left-hand side and using the $\FF_q[\ut_s]$-linearity of $\exp_{\phi}$, we see that $X_0/\theta^{n_0} - \sum_{i=1}^m b_{0,i}\lambda_i/\theta = X_1$ for some $X_1\in \Lambda_{\phi}$ and $\dnorm{X_1}\leqslant \sup\{\dnorm{X_0/\theta^{n_0}}, \dnorm{\lambda_1/\theta}, \dots, \dnorm{\lambda_m/\theta} \}$.  If $X_1=0$, then we are done.  Otherwise, we continue in a similar fashion to produce $X_1, \dots, X_k \in \Lambda_{\phi}$, together with $n_1, \dots, n_k \geqslant 1$ and $b_{j,i} \in \FF_q[\ut_s]$ for $1 \leqslant j \leqslant k$, $1 \leqslant i \leqslant m$, so that
\begin{equation}\label{E:partialsums}
X_0=\theta^{n_0+n_1+\dots + n_k}X_{k+1}+ \theta^{n_0+n_1+\dots + n_k-1}\sum_{i=1}^m b_{k,i}\lambda_i  + \dots + \theta^{n_0-1} \sum_{i=1}^m b_{0,i}\lambda_i,
\end{equation}
and
\begin{equation}\label{E:normbound}
\dnorm{X_{k+1}}\leqslant \sup \biggl\{ \biggl\lVert \frac{X_0}{\theta^{n_0+n_1+\dots+n_k}} \biggr\rVert_{\infty}, \biggl\lVert\frac{\lambda_1}{\theta}\biggr\rVert_{\infty}, \dots, \biggl\lVert\frac{\lambda_m}{\theta}\biggr\rVert_{\infty} \biggr\}.
\end{equation}
Eventually we will find $k \geqslant 1$ so that for all $1 \leqslant i \leqslant m$,
\begin{equation}\label{E:rofconv}
\biggl\lVert \frac{X_0}{\theta^{n_0+n_1+\dots+n_k}}\biggr\rVert_{\infty} \leqslant \biggl \lVert \frac{\lambda_i}{\theta}\biggr\rVert_{\infty} < \varepsilon_{\phi}.
\end{equation}
In this case \eqref{E:normbound} and \eqref{E:rofconv} imply that $\log_{\phi}$ converges at $X_{k+1}$. Since $X_{k+1}\in \Lambda_{\phi}$, we have that $\exp_{\phi}(X_{k+1})=0$, but Lemma~\ref{L:iso} implies that $\exp_{\phi}$ is injective on the open disk of radius $\varepsilon_{\phi}$, and so $X_{k+1}=0$.
\end{proof}

To produce non-isotrivial uniformizable Drinfeld $A[\ut_s]$-modules, we appeal to techniques in~\cite[\S 6]{EP14}.  Let $\phi$ be a Drinfeld $A[\ut_s]$-module of rank $r$ defined by $\phi_{\theta}=\theta+A_1\tau + \dots + A_r\tau^r$ and let $k_{\phi}$ be the smallest index such that
\[
\frac{\Ord(A_{k_{\phi}})+q^{k_{\phi}}}{q^{k_{\phi}}-1} \leqslant \frac{\Ord(A_j)+q^j}{q^j-1}
\]
for all $j$ such that $A_j\neq 0$. For any $n\geqslant 0$, we recall the set $P_r(n)$ from \S\ref{SS:ExpLog} and define
\[
\gamma_n(z) = \sum_{(S_1, \dots, S_r) \in P_r(n)} \prod_{i=1}^r \prod_{j \in S_i} \frac{\tau^j(A_i)}{z-\theta^{q^{i+j}}} \in \TTsz.
\]
These functions serve similar purposes as the functions $\mathcal{B}_n(t)$ in~\cite[Eq.~(6.4)]{EP14}, and in particular via \eqref{E:logcoeffs} they specialize at $z=\theta$ as logarithm coefficients:
that is, $\gamma_n(\theta) = \beta_n$, where $\log_{\phi} = \sum_{n \geqslant 0} \beta_n \tau^n \in \TTs[[\tau]]$.  Similar calculations as in \cite[Lem.~6.7(b)]{EP14}, yield that for $f \in \TTs$ with $\Ord(f) > -q$,
\begin{equation}\label{order}
\Ord(\gamma_n(f)) \geqslant \frac{q^n-1}{q^{k_{\phi}}-1}(\Ord(A_{k_{\phi}})+q^{k_{\phi}}).
\end{equation}

\begin{proposition}[{cf.~\cite[Prop.~6.10]{EP14}}]\label{P:logarithm}
Let $C_{\phi}=-(\Ord(A_{k_{\phi}})+q^{k_{\phi}})/(q^{k_{\phi}}-1)$ and $f\in\TTs$. If $\Ord(f)>C_{\phi}$, then $\log_{\phi}(f)$ converges in $\TTs$.
\end{proposition}

\begin{proof}
We combine the above considerations with the property that $\log_{\phi}(f)=\sum_{n \geqslant 0} \beta_n f^{(n)}$ converges if and only if $\Ord(\beta_n f^{(n)}) \to \infty$ as $i \to \infty$.
\end{proof}

We now produce a class of non-isotrivial uniformizable Drinfeld $A[t_1]$-modules.  However, it should be noted that the same techniques can be used to produce additional examples in more variables.   The proof of Proposition~\ref{P:finalresult} occupies the rest of this section.

\begin{proposition} \label{P:finalresult}
For the Drinfeld $A[t_1]$-module $\phi$ defined by $\phi_{\theta}=\theta + t_1\tau +\tau^r$ for $r\geqslant 2$,
\begin{enumerate}
\item[(a)] $\phi$ is non-isotrivial,
\item[(b)] $\phi$ is uniformizable,
\item[(c)] $\Lambda_\phi$ is free of rank $r$ over $A[t_1]$.
\end{enumerate}
Moreover, the equivalent statements of Theorem~\ref{T:characterization} all hold for $\phi$.
\end{proposition}

\begin{lemma}\label{elements}
If $g\in \phi[\theta] \subseteq \TT_1$, then $g$ is within the radius of convergence of $\log_{\phi}$.
\end{lemma}

\begin{proof}
Let $g=\sum_{i \geqslant 0} b_i t_1^i \in \phi[\theta]$. Then
\begin{equation}\label{E:radofconv}
\sum_{i=0}^{\infty} \bigl(\theta b_i+b_i^{q^r} \bigr) t_1^i + \sum_{i=1}^{\infty} b_{i-1}^{q}t_1^{i} = 0.
\end{equation}
Observing that $k_{\phi}=r$, Proposition~\ref{P:logarithm} implies that it suffices to show that $\Ord(g) > -q^r/(q^r-1)$.  If we compare coefficients of $t_1$ on both sides of \eqref{E:radofconv}, then we have that $b_0$ is a root of the polynomial $v_0(x)=\theta x + x^{q^r}$.  The Newton polygon of $v_0$ shows that it has $q^r-1$ non-zero roots with valuation $-1/(q^r-1)$, and thus all roots have valuations greater than $-q^r/(q^r-1)$.  Consider the polynomial $v_1(x) := b_0^q + \theta x + x^{q^r}$, which has $b_1$ as a root by~\eqref{E:radofconv}.  If we choose $b_0=0$, then $v_1$ turns out to be the polynomial $v_0$.  Letting $\ord_{\infty}(b_0) = -1/(q^r-1)$, the Newton polygon of $v_1$ shows that it has $1$ root with valuation $-q/(q^r-1)+1$ and $q^r-1$ roots with valuation $-1/(q^r-1)$.  Again, every root has valuation greater than $-q^r/(q^r-1)$.  Proceeding by induction, we find that for $i \geqslant 0$, we always have $\ord_{\infty}(b_i) > -q^r/(q^r-1)$.  Thus, $\Ord(g) = \inf(\ord_{\infty}(b_i)) > -q^r/(q^r-1)$.
\end{proof}

Let $D$ be the constant Drinfeld module defined by $D_{\theta} = \theta + \tau^r$, and let $\tpi_1, \dots, \tpi_r \in \CC_{\infty}$ form an $A$-basis for $\Lambda_D$.  For $1\leqslant j \leqslant r$, let $a_{j,0} = \exp_{D} (\tilde{\pi}_j/\theta)$, and for all $i\geqslant 0$, let $a_{j,i+1}$ be the root of the polynomial $a_{j,i}^{q} + \theta X + X^{q^r}$ that has the maximum valuation among all its roots.  That is, for $i\geqslant 0$, we find $\ord_{\infty}(a_{j,i+1}) = -q^{i+1}/(q^r-1) + 1+q+q^2 + \dots + q^i$. We remark that for $i \geqslant 0$ and $1\leqslant j \leqslant r$, the existence and uniqueness of $a_{j,i+1}$ are guaranteed by considering Newton polygons. Now set $F_j := \sum_{i=0}^{\infty}a_{j,i}t_1^i \in \TT_1$.

\begin{proposition}\label{P:rank}
The $\FF_q[t_1]$-module $\phi[\theta]$ is free of rank $r$ with basis $F_1,\dots, F_r$.
\end{proposition}

\begin{proof}
Note that for $1\leqslant j \leqslant r$, $F_j \in \phi[\theta]$ by the construction.  Moreover, $F_1, \dots, F_r$ are $\FF_q[t_1]$-linearly independent because any linear dependency would contradict the fact that $\exp_{D}(\tilde{\pi}_1/\theta),\dots ,\exp_{D}(\tilde{\pi}_r/\theta)$ are $\FF_q$-linearly independent.

Let $g = \sum b_i t_1^i \in \phi[\theta] \subseteq \TT_1$. Thus, the coefficients $b_i$ satisfy the same recursions as in~\eqref{E:radofconv}.  As in that case $D_{\theta}(b_0) = \theta b_0 + b_0^{q^r} = 0$. Since $b_0\in D[\theta]$, there exists $c_{j,0}\in \FF_q$ for $1\leqslant j\leqslant r$ such that $b_0 = \sum_j c_{j,0}a_{j,0}$.  Moreover, comparing coefficients of $t_1$ on both sides of \eqref{E:radofconv} implies that $\theta b_1 + b_1^{q^r} + b_0^q=0$. Since the polynomial $\theta X + X^{q^r} + b_0^q$ has no repeated roots, each root can be written as $y+w$, where $y\in \CC_{\infty}$ is the root which has the maximum valuation among the other roots and $w\in \CC_{\infty}$ is any root of the polynomial $\theta X + X^{q^r}$.  Therefore, there exist $c_{j,1} \in \FF_q$ such that $b_1=\sum_j (c_{j,0}a_{j,1} + c_{j,1}a_{j,0})$. Similarly, using \eqref{E:radofconv} we find for all $n \geqslant 0$ that there exist $c_{j,n} \in \FF_q$ such that
\begin{equation}\label{E:generates2}
b_n = \sum_{k=0}^n \sum_{j=1}^r c_{j,n-k} a_{j,k}.
\end{equation}
We find after some calculation that
\[
g = \sum_{j=1}^r \biggl( \sum_{k=0}^\infty a_{j,k} t_1^k \biggr) \biggl( \sum_{n=0}^{\infty} c_{j,n} t_1^n \biggr)
= \sum_{j=1}^r \biggl( \sum_{n=0}^{\infty} c_{j,n} t_1^n \biggr) F_j \in \Span_{\FF_q[[t_1]]}(F_1, \dots, F_r).
\]
If for arbitrarily large $n$ we can always find $c_{j,n} \neq 0$, then after a short argument \eqref{E:generates2} implies that $|b_n|_{\infty} \nrightarrow 0$ as $n \to \infty$.  This contradicts the choice of $g \in \TT_1$, and so we find that $\phi[\theta] \subseteq \Span_{\FF_q[t_1]} ( F_1, \dots, F_r)$.
\end{proof}

\begin{proof}[Proof of Proposition~\ref{P:finalresult}]
First, we see that $\phi$ is non-isotrivial because $t_1 \notin \TT_1^{\times}$.  Second, by Lemma~\ref{elements}, all of $\phi[\theta]$ is within the radius of convergence of $\log_{\phi}$, and by Proposition~\ref{P:rank}, we know that $\phi[\theta]$ is a free $\FF_q[t_1]$-module of rank $r$. Therefore, the result follows from Theorem~\ref{T:characterization} and \ref{T:RANKROC}.
\end{proof}

\end{document}